\documentclass[12pt]{article}

\DeclareUnicodeCharacter{0301}{{\color{red}\'{e}}}

\usepackage{cmap}  
\usepackage[T2A]{fontenc}
\usepackage[utf8]{inputenc}
\usepackage[english]{babel}
\usepackage[backend=biber,
style=alphabetic, sorting=nyt, isbn=false, date=year]{biblatex}
\AtEveryBibitem{\clearlist{language} \clearfield{urlyear}}

\usepackage[usenames]{color}
\usepackage{amsmath,amssymb,amsthm}
\usepackage[colorlinks=true,linkcolor=blue,urlcolor=red,unicode=true,hyperfootnotes=false,bookmarksnumbered]{hyperref}
\usepackage{mathtools}
\usepackage{xcolor}
\usepackage{csquotes}
\usepackage{wrapfig}
\usepackage{setspace}
\usepackage{enumitem}
\usepackage{appendix}
\usepackage{algorithm}
\usepackage{algpseudocode}
\usepackage{dsfont}
\usepackage{comment}
\usepackage{bbding}
\usepackage[nomessages]{fp}
\usepackage{tikz}

\newcommand{\ff}{\mathcal{F}}
\newcommand{\hh}{\mathcal{H}}
\newcommand{\cG}{\mathcal{G}}
\newcommand{\cS}{\mathcal{S}}
\newcommand{\cR}{\mathcal{R}}
\newcommand{\cA}{\mathcal{A}}
\newcommand{\cB}{\mathcal{B}}
\newcommand{\cQ}{\mathcal{Q}}

\newcommand{\cT}{\mathcal{T}}

\newcommand{\cD}{\mathcal{D}}
\newcommand{\cC}{\mathcal{C}}
\newcommand{\cE}{\mathcal{E}}

\newcommand{\cO}{\mathcal{O}}

\newcommand{\bfS}{\mathbf{S}}

\renewcommand{\ge}{\geqslant}
\renewcommand{\le}{\leqslant}
\renewcommand{\mid}{\,:\,}

\newcommand{\diam}{\operatorname{diam}}

\floatname{algorithm}{Procedure}

\newcommand{\bbC}{\mathbb{C}}

\DeclareMathOperator{\support}{supp}

\DeclareUnicodeCharacter{02BC}{\textcolor{red}{U}}

\newcounter{assumptionCounterNLarge}
\renewcommand{\theassumptionCounterNLarge}{\arabic{assumptionCounterNLarge}$'$}
{}

\newtheorem{theorem}{Theorem}[section]
\newtheorem{proposition}[theorem]{Proposition}
\newtheorem{lemma}[theorem]{Lemma}
\newtheorem{corollary}[theorem]{Corollary}

\newtheorem{claim}[theorem]{Claim}

\newtheorem{example}{Example}
\newtheorem{problem}{Probem}

\title{Exact results and the structure of extremal families for the Duke--Erd\H{o}s forbidden sunflower problem}
\author{Andrey Kupavskii and Fedor Noskov}

\begin{document}

\maketitle
\begin{abstract}
    In 1977, Duke and Erd\H os asked the following general question:  What is the largest size of a family $\ff \subset \binom{[n]}{k}$ that does not contain a sunflower with $s$ petals and core of size exactly $t - 1$? This problem is closely related to the famous Erd\H os--Rado sunflower problem of determining the size $\phi(s,t)$ of the largest $t$-uniform family with no $s$-sunflower. In this paper, we answer this question {\em exactly} for $t=2$, odd $s$ and $k\ge 5$, provided $n$ is large enough. Previously, the only know exact extremal result on this problem was due to Chung and Frankl from 1987. 

    One of the important ingredients for the proof that we obtained is a stability result for the Duke--Erd\H os problem, which was previously not known, mostly due to our lack of understanding of the behaviour of $\phi(s,t)$. 

    For large $k$ and $n$ we in fact manage to reduce the Duke--Erd\H os problem to an Erd\H os--Rado-like problem which depends on $t$ and $s$ only. In particular, we get a good understanding of the structure of extremal families for the Duke--Erd\H os problem in terms of the Erd\H os--Rado problem. Previously, a much looser variant of this connection (only in terms of the sizes, rather than the structure, of respective extremal families) was established in a seminal work of Frankl and F\"uredi from 1987.   
\end{abstract}

\section{Introduction}

Let $[n]$ stand for the standard set $\{1,\ldots, n\}$ and, given a set $X$ and an integer $k$, we denote by $2^X$ ($\binom{X}{k}$) the family of all subsets of $X$ (all subsets of $X$ of size $k$). A family of $s$ sets $\{F_1, \ldots, F_s\}$ is a {\it $\Delta(s)$-system} or a \textit{sunflower with $s$ petals} if, for any distinct $i, j \in [s]$, one has $F_i \cap F_j = \cap_{u = 1}^s F_u$. The set $C = \cap_{u = 1}^s F_u$ is called the {\it core} or {\it kernel} of the sunflower. 

One of the most famous problems in Extremal Set Theory is the  Erd\H os--Rado \cite{erdos1960intersection}  sunflower problem.
\begin{problem}[Erd\H os and Rado, 1960]
\label{problem: erdos sunflower problem}
What is the maximum size $\phi(s, t)$ of a family $\cT$ that consists of sets of size $t$ and does not contain a sunflower with $s$ petals? Show that $\phi(s,t) \le (Cs)^t$ for some absolute constant $C.$
\end{problem}
It is one of Erd\H os' $1000$\$ problems. Erd\H os and Rado proved the bounds  bounds $(s-1)^t\le \phi(s,t)\le (t-1)!(s-1)^t$. Over a long period of time, the best known upper bound on $\phi(s,t)$ was of the form $t^{t (1 + o(1))}$, where the $o(1)$ term depends on $s$, until the breakthrough result of Alweiss, Lovett, Wu and Zhang~\cite{alweiss2021improved}. After a series of further improvements in a series of papers and blogposts~\cite{rao_coding_2020,tao_sunflower_2020,bell_note_2021, hu_entropy_2021, stoeckl_lecture_nodate, rao2023sunflowers}, the bound is $\phi(s, t) \le (C s \ln t)^t$ for some absolute constant $C$. The exact value of $\phi(s,t)$ is unknown for most non-trivial cases, with a notable exception of the case $\phi(s,2)$, where an exact result was obtained by Abbott, Hanson and Sauer~\cite{abbott1972intersection}. Their result is stated in  Theorem~\ref{theorem: sunflower graph case} below.

In 1977, Duke and Erd\H os~\cite{duke1977systems} asked a closely related question.

\begin{problem}[Duke and Erd\H os, 1977]
\label{problem: sunflower problem}
Suppose that a family $\ff \subset \binom{[n]}{k}$ does not contain a sunflower with $s$ petals and core of size exactly $t - 1$. What is the maximum size of $\ff$?
\end{problem}
The cases  $t \in \{2, 3\}$, $k = 3$ and arbitrary $s$  were studied by Duke and Erd\H{o}s~\cite{duke1977systems}, Frankl~\cite{frankl1978extremal}, Chung~\cite{chung1983unavoidable} and Chung and Frankl~\cite{chung1987maximum}. In the last paper, Chung and Frankl managed to determine the exact maximum size of a family of $3$-element sets ($3$-sets) that avoids a sunflower with $s$ petals and core of size $1$ and large $n$. To the best of our knowledge, so far this was the only known {\em exact} result for this problem for $s,t\ge 3$. 

The first main result of this paper is an {\em exact} result for families of larger uniformity that avoid sunflowers with core of size $1$, which extends the result of Chung and Frankl \cite{chung1987maximum}. 

\begin{theorem}
\label{theorem: graph case Duke-Erdos extremal}
Let $s$ be odd and $k \ge 5$ and $n\ge n_0(s,k)$. Let $\ff \subset \binom{[n]}{k}$  be an extremal family that does not contain a sunflower with $s$ petals and the core of size $1$. Then, there exists a graph $G = K_1 \sqcup K_2$ consisting of two disjoint cliques $K_1, K_2$ of size $s$, such that
\begin{align*}
    \ff = \bigg \{F \in \binom{[n]}{k} \mid & |F \cap V(G) | \ge 2 \\
    & \text{ and } \forall i \in \{1,2\} \text{ we have } |F \cap V(K_i)| \neq 1 \bigg \}.
\end{align*}
\end{theorem}

Here the graph $G$ is coming from the main result of \cite{abbott1972intersection} . It is an extremal $2$-uniform family without a sunflower with $s$ petals (i.e., a graph without a matching of size $s$ or a vertex of degree $s$). It is  unique up to isomorphism. If $s$ is even and larger than $2$, then an extremal example in the Erd\H{o}s sunflower problem for $t = 2$ is not unique, see Theorem~\ref{theorem: sunflower graph case} below. This makes it technically much more demanding to state and prove its analogue for even $s.$

One of the important tools in this area and for this paper is the Delta-system method (see the recent survey of the first author~\cite{kupavskii2025delta}). It was greatly developed by Frankl and F\"uredi in late 70s and 80s~\cite{frankl_forbidding_1985,frankl_forbidding_1985}. Using the delta-system method,  Frankl and F\"uredi in their seminal paper~\cite{Frankl1987} derived the following asymptotically tight bound on a sunflower-free  family $|\ff|$ when $s, k, t$ are fixed and $n$ tends to infinity. 

\begin{theorem}[Frankl and F\"uredi, 1987]
\label{theorem: asymptotic bound due to Frankl}
Let $s, k, t$ be integers, $k \ge 2t + 1$. Suppose that $s, k, t$ are fixed and $n\to \infty$. Then
\begin{align*}
    |\ff| \le (\phi(s, t) + o(1)) \binom{n}{k - t},
\end{align*}
where $\phi(s, t)$ is the maximal size of a $t$-uniform family $\cT$ that does not contain a sunflower with $s$ petals (and arbitrary core).
\end{theorem}

One of the major complications for the study of this problem is that we know little about the behavior of the function $\phi(s, t)$. However, Frankl and F\"uredi managed to prove their Theorem~\ref{theorem: asymptotic bound due to Frankl}  even without much understanding of the behaviour of this function.

Recently, the case $k = 4$ was studied by Buci{\'c} et al. in~\cite{bucic2021unavoidable} in the regime when we allow $s$ to grow with $n$. Later, Brada{\v{c}}, Buci{\'c} and Sudakov~\cite{bradavc2023turan} generalized the previous result and proved that an extremal family in Problem~\ref{problem: sunflower problem} has size $O_k(n^{k - t} s^t)$ when $n$ tends to infinity, $s$ grows arbitrarily with $n$, and $k$ is fixed. The complementary scenario (when $k$ can grow linearly with $n$, and $s,t$ are small) was studied by the authors in~\cite{kupavskii2025lineardependenciespolynomialfactors}, where the authors of this paper managed to determine the  solution to the problem up to a mutliplicative factor of $(1+o(1))$. 

Problem~\ref{problem: sunflower problem} generalizes two famous problems proposed by Erd\H{o}s and coauthors, namely Erd\H{o}s--S\'os forbidden intersection problem~\cite{erdos1975problems} and Erd\H{o}s Matching Conjecture~\cite{erdos1965problem}. The Erd\H{o}s--S\'os problem has numerous applications in discrete geometry~\cite{frankl1990partition, frankl1981intersection}, communication complexity~\cite{sgall1999bounds} and quantum computing~\cite{buhrman1998quantum}. Attempts to solve it contributed to development and improvements of a large variety of combinatorial methods, including the delta-systems method~\cite{frankl_forbidding_1985, janzer2025sunflowers}, the junta method~\cite{keller2021junta, ellis2024stability, ellis2023forbidden}, the spread approximation technique~\cite{kupavskii_spread_2022, kupavskii2025lineardependenciespolynomialfactors}, and the hypercontractivity approach~\cite{keevash2023forbidden}. Some partial cases of this problem were studied in~\cite{frankl1981intersection, cherkashin2024set, keevash2006set}. Meanwhile, the Erd\H{o}s Matching Conjecture (EMC) has applications in concentration of measure theory~\cite{alon2012nonnegative,alon2012large_matchings} and distributed memory allocation~\cite{alon2012large_matchings}. Great progress on EMC was made in papers~\cite{bollob1976sets,huang2012size,frankl2012matchings,frankl_improved_2013, frankl2017proof, frankl2022erdHos,kolupaev2023erdHos}. Let us  also  mention a recent generalization of Problem~\ref{problem: sunflower problem}, proposed by authors of~\cite{janzer2025sunflowers}, when the size of the core of a sunflower is allowed to take values in some set $L$. This generalization captures a well-known problem of $(n, k, L)$-systems~\cite{deza_intersection_1978}.

Theorem~\ref{theorem: asymptotic bound due to Frankl} is essentially sharp, as the following example due to Frankl and F\"uredi demonstrates.

\begin{example}[{\cite[Example 2.3]{Frankl1987}}]
\label{example: basic example}
Suppose that $n > t \phi(s, t)$. Choose a family $\cT \subset \binom{[n]}{t}$ that does not contain a sunflower with $s$ petals, and $|\cT| = \phi(s, t)$. Then, define a family $\ff$ as follows:
\begin{align*}
    \ff = \left \{F \in \binom{[n]}{k} \mid F \cap \support(\cT) \in \cT \right \}.
\end{align*}
The family $\ff$ does not contain a sunflower with $s$ petals and the core of size $t - 1$.
\end{example}

Indeed, the example provides a family $\ff$ without a sunflower with $s$ petals and the core of size at most $t - 1$ such that
\begin{align*}
    |\ff| & \ge \phi(s, t) \binom{n - t \phi(s, t)}{k - t} =(1+o(1))
   \phi(s, t) \binom{n}{k - t}.
\end{align*}

Although Frankl and F\"uredi managed to obtain the asymptotic solution to the Duke--Erd\H os problem in the regime $n>n_0(s,k),$ they did not prove any results concerning the structure of extremal results. A big obstacle in that respect is that a stability result was missing. Here, the situation is in contrast with the aforementioned Erd\H os--S\'os problem. Let us formulate the corresponding extremal question: what is the largest size of the family $\ff\subset {[n]\choose k}$ that has no two sets $A,B$ with $|A\cap B| = t-1$ (i.e., a family that avoids sunflowers with $2$ petals and core of size $t-1$)? There, we should mention a result of Keevash, Mubayi, and Wilson \cite{keevash2006set}, who proved a stability result for the Erd\H os-S\'os problem for $t=1$. Actually, a stability argument is implicitly present in the original paper of Frankl and F\"uredi \cite{frankl_forbidding_1985}, where they used the delta-system method to find an exact solution to the problem for $k\ge 2t-1$ and $n\ge n_0(k)$. For the discussion of stability in this setting, we refer to Section 6 of the aforementioned survey \cite{kupavskii2025delta} of the first author. 

The difficulty of the Duke--Erd\H os problem is again in our poor understanding of $\phi(s,t)$ and the corresponding extremal families. In this paper, we fill this gap and prove a stability result. The following result can be thought of as  a `first-order' stability for  this problem.

\begin{theorem}
\label{theorem: stability in Erdos-Duke}
Let $k \ge 2t + 1$. Let $\ff$ be a family that does not contain a sunflower with $s$ petals and the core of size $t - 1$. Suppose that $|\ff| \ge (1 - \delta) \phi(s, t) \binom{n}{k - t}$ for some $\delta > 0$. Then, there exist a family $\cT \subset \binom{[n]}{t}$ that does not contain a sunflower with $s$ petals and a constant $C(s, k)$ depending on $s,k$ only, such that
\begin{align*}
    |\ff\setminus\ff[\cT]| \le \left (3k\phi^2(s, t) \cdot  \delta + \frac{C(s, k)}{n} \right ) \cdot \phi(s, t) \binom{n - t}{k - t},
\end{align*}
provided $n \ge n_0(s, k)$ for some function $n_0(s, k)$.
\end{theorem}

The idea behind the proof of Theorem~\ref{theorem: stability in Erdos-Duke} is based on expansion of the Johnson graph on $\binom{[n]}{k - t}$, which can be thought as an action of a certain random walk on the shadow of $\ff$. Related techniques were used previously in  Extremal Set Theory, cf. the seminal proof of the Erd\H{o}s--Ko--Rado theorem by Lov\'asz~\cite{lovasz1979shannon}, application of Along and Chung bound~\cite{alon1988explicit} to the analysis of the Erd\H{o}s Matching Conjecture~\cite{frankl2022erdHos},  hypercontracitivty approach of Keevash, Lifshitz, Keller and coathors~\cite{keevash2021global, keevash2023forbidden} based on the action of a certain semigroups or  proofs of the spread lemma~\cite{mossel2022second, tao_sunflower_2020, stoeckl_lecture_nodate} based on certain Markov chains. However, as our paper shows, the connection between the hypergraph Turan-type problems and the semigroup/Markov chains theory is still under development. Because of the lack of our knowledge on Problem~\ref{problem: erdos sunflower problem}, the techniques that we use are not that specific for Duke--Erd\H os Problem~\ref{problem: sunflower problem}. In our further research, we plan to apply it to other Turan-type problems, combining with other available techniques.

Theorem~\ref{theorem: stability in Erdos-Duke} is the starting point for iteratively getting higher-order stability for extremal $|\ff|$ that gradually reveals the structure of extremal families. Building on this theorem, we prove the following theorem, which guarantees in some regimes of the parameters that any extremal family $\ff$ must in fact contain the family from Example~\ref{example: basic example}.
\begin{theorem}
\label{theorem: structural duke-erdos}
Let $k \ge \max\{2t + 1, 2 t (t - 1)\}$ and $n\ge n_0(s,k)$. Let $\ff \subset \binom{[n]}{k}$ be an extremal family that does not contain a sunflower with $s$ petals and the core of size $t - 1$. Then, there exists a family $\cT \subset \binom{[n]}{t}$ that does not contain a sunflower with $s$ petals such that $|\cT| = \phi(s,t)$ and
\begin{align*}
    \left \{F \in \binom{[n]}{k} \mid F \cap \support(\cT) \in \cT \right \} \subset \ff.
\end{align*}
In particular, for each $F \in \ff$, we have either $T \subset F$ for some $T \in \cT$ or $|F \cap \support \cT| \ge t + 1$.
\end{theorem}
The appearance of $\phi(s, t)$ and $\mathcal T$ in  Theorems~\ref{theorem: asymptotic bound due to Frankl},~\ref{theorem: structural duke-erdos} and Example~\ref{example: basic example} suggests that there may be a `black-box' reduction from a `high-dimensional' Duke--Erd\H os Problem~\ref{problem: sunflower problem} to the `low-dimensional' Problem~\ref{problem: erdos sunflower problem}. It turns out that is is not true {\em exactly}, but {\em approximately}. The Problem~\ref{problem: sunflower problem} is rather a `first-order approximation' of the appropriate low-dimensional problem.  There is indeed an extremal problem  independent of $n,k$ such that the Duke--Erd\H os problem can be reduced to it, provided $n,k$ are large enough. We present it in Section~\ref{section: extremal families}. In that same section we present an exact extremal result (Theorem~\ref{theorem: exact erdos-duke}) for the Duke--Erd\H os problem in terms of this `low-dimensional' problem. In fact, Theorem~\ref{theorem: graph case Duke-Erdos extremal} can be seen as a particular instance of this general result with a clean structure of the extremal example.

The paper is organized as follows. In Section~\ref{section: notation}, we introduce some additional notation. In Section~\ref{section: extremal families}, we describe the extremal families of Problem~\ref{problem: sunflower problem} in terms of a  a certain extremal problem independent of $n, k$. In Section~\ref{section: tools}, we list ancillary results used in our proofs. In Section~\ref{section: stability of ED -- proof}, we prove our main technical  Lemma~\ref{lemma: induction over layers} and deduce Theorem~\ref{theorem: stability in Erdos-Duke} from it. In Section~\ref{section: proof of structural duke-erdos}, we prove Theorem~\ref{theorem: structural duke-erdos}. In Sections~\ref{section: proof of theorem -- graph case} and~\ref{section: proof of the exact duke-erdos}, we prove Theorem~\ref{theorem: graph case Duke-Erdos extremal} and Theorem~\ref{theorem: exact erdos-duke} respectively.

\subsection{Notation}
\label{section: notation}
Given a set $X$ and an integer $h$, we denote by $\binom{X}{h}$ the family of all subsets of $X$ of size $h$. 

For a family $\ff$ and sets $A\subset B$, we define
\begin{align*}
    \ff(A, B) = \{F \setminus B \mid F \in \ff \text{ and } F \cap B = A\}. 
\end{align*}
If $A = B$, we denote $\ff(B) = \ff(B, B)$. If $B$ is a singleton $\{x\}$, we omit braces and use $\ff(x) = \ff(\{x\})$. 
Next, if $\cB$ is a family of sets, then we define
\begin{align*}
    \ff[\cB] = \{F \in \ff \mid \exists B \in \cB \text{ such that } B \subset F\}.
\end{align*}
If $\cB = \{B\}$ is a singleton of a set $B$, we omit braces and use $\ff[B] = \ff[\{B\}]$. We define the support of a family $\ff$ as $\support(\ff) = \bigcup_{F\in \ff}F$. For families $\ff$ and $\cB$, we define
\begin{align*}
    \ff \vee \cB = \{F \cup B\mid F \in \ff \text{ and } B \in \cB\}.
\end{align*}

A family $\ff$ is  {\it $k$-uniform} ({\it $\le k$-uniform}), if all sets in $\ff$ have size exactly $k$ (at most $k$). 
Given a family $\ff$ and a number $h$, we define its $h$-th layer $\ff^{(h)}$ as follows:
\begin{align*}
    \ff^{(h)} = \{F \in \ff \mid |F| = h \}.
\end{align*}
The {\it shadow} of a family $\ff$ on the layer $h$ is defined as follows:
\begin{align*}
    \partial_h \ff = \bigcup_{F \in \ff} \binom{F}{h}.
\end{align*}
Given numbers $n, k, t$ and a family $\cS \subset \bigcup_{q = t}^n \binom{[n]}{q}$, we define a family $\ff_{\cS} \subset \binom{[n]}{k}$ as follows:
\begin{align*}
    \ff_{\cS} = \left \{ F \in \binom{[n]}{k} \mid F \cap \support \cS^{(t)} \in \cS \right \}.
\end{align*}

We are working in the regime when $k,s$ are fixed and $n\to\infty$, and so it is convenient to use the $O$-notation. For two functions $f_1(n, s,k)$ and $f_2(n, s, k)$, we say $f_1 = O_{s,k}(f_2)$ if for any fixed $s, k$ there are some constants $C=C(s,k), n_0(s,k)$ such that 
\begin{align*}
    |f_1(n, s, k)| \le C |f_2(n, s, k)|
\end{align*}
for any $n \ge n_0$.

\section{Extremal families}
\label{section: extremal families}

In this section, we state a result that shows that the Duke--Erd\H os problem for large $k=k(s,t)$ and $n=n(k,s,t)$ can be reduced to another extremal problem that is independent of $n, k$. First, we describe the corresponding extremal problem, together with an extremal family. 

Denote $T = t \cdot \phi(s, t)$. Throughout this paper, we use $T$ as the universal upper bound of $\support \cT$ for any $t$-uniform family $\cT$ without a sunflower of size $s$. Let $\cS_*$ be a family with the following properties:
\begin{enumerate}
    \item \label{item: sunflower freeness S*}  $\cS_*$ does not contain a sunflower with $s$ petals and the core of size at most $t - 1$;
    \item \label{item: lower levels emptiness of S*} $\cS_* \subset \bigcup_{q = t}^T \binom{\support \cS^{(t)}_*}{q}$;
    \item \label{item: lexcicographic maximality S*}  a vector-valued function $\widetilde{\phi}(\cS_*) = (\widetilde{\phi}_0(\cS_*), \ldots, \widetilde{\phi}_{T - t}(\cS_*))$ which components are defined as
    \begin{align*}
        \widetilde{\phi}_0(\cS) & = |\cS^{(t)}|, \\
        \widetilde{\phi}_i(\cS) & = |\cS^{(t + i)}| + \sum_{j = 0}^{i - 1} |\cS^{(t + j)}| \binom{T - |\support \cS^{(t)}|}{i - j};
    \end{align*}
    is maximal in the lexicographical order over all choices of $\cS_*$ satisfying properties~\ref{item: sunflower freeness S*}-\ref{item: lower levels emptiness of S*}. Here we assume that $\binom{m}{a} = 0$ if $m < a$.
\end{enumerate}
Then, define
\begin{align}
\label{eq: extremal example refined}
\ff_* = \left \{F \in \binom{[n]}{k} \mid F \cap \support \cS^{(t)}_* \in \cS_* \right  \}.
\end{align}
Let us show that $\ff_*$ does not contain a sunflower with $s$ petals and the core of size $t - 1$. Suppose that such a sunflower $F_1, \ldots, F_s \in \ff$ exists. Then, $F \cap \support \cS^{(t)}_*, \ldots, F_s \cap \support \cS^{(t)}_*$ form a sunflower with $s$ petals and the core of size at most $t - 1$. But $F_i \cap \support \cS^{(t)}_* \in \cS_*$, contradicting property~\ref{item: sunflower freeness S*}.

Naturally, we know very little of the possible structure of $\cS^*$, except in some special cases. It is well possible that the family $\cS_*$ is not unique. However, we are still able to prove the following theorem. 

\begin{theorem}
\label{theorem: exact erdos-duke}
Let $k \ge (t + T)(t - 1)$. Let $\ff \subset \binom{[n]}{k}$ be an extremal family that does not contain a sunflower with $s$ petals and the core of size $t - 1$. Then, $\ff = \ff_*$ for some $\cS_*$ satisfying properties~\ref{item: sunflower freeness S*}-\ref{item: lexcicographic maximality S*}, provided $n$ is large enough as a function of $s, k$.
\end{theorem}
The proof of the theorem is presented in Section~\ref{section: proof of the exact duke-erdos}.

Let us comment on the property~\ref{item: lexcicographic maximality S*}. As a part of our analysis, we will prove that for any extremal family $\ff$ there is a family $\cS$ without a sunflower with $s$ petals and the core of size at most $t - 1$ such that $\ff$ is almost fully contained in $\ff_{\cS}$, see Lemma~\ref{lemma: induction over layers}. The representation
\begin{align*}
    |\ff_{\cS}| = \sum_{i = 0}^{T - t} |\cS^{(t + i)}| \binom{n - |\support \cS^{(t)}|}{k - t -  i},
\end{align*}
is not convenient for describing extremal families as $|\support \cS^{(t)}|$ may vary for different $\cS$. To overcome this obstacle, WLOG assume for a while that $\support \cS^{(t)} \subset [T]$ and define
\begin{align*}
    \cG^i_{\cS} = \left \{ G \in \binom{[T]}{t + i} \mid G \cap \support \cS^{(t)} \in \cS \right \}.
\end{align*}
Then, we clearly have $|\cG^i_{\cS}| = \widetilde{\phi}_{i}(\cS)$ and
\begin{align*}
    |\ff_{\cS}| = \sum_{i = 0}^{\min\{T - t, k - t\}} |\cG^i_{\cS}| \binom{n - T}{k - i - t} = \sum_{i = 0}^{\min\{T - t, k - t\}} \widetilde{\phi}_{i}(\cS) \binom{n - T}{k - i - t}.
\end{align*}
This yields two simple, but important corollaries. First, it implies that for any $\cS_*$ satisfying~\ref{item: sunflower freeness S*}-\ref{item: lexcicographic maximality S*}, $\ff_{*}$ defined by~\eqref{eq: extremal example refined} is an extremal family, provided $k \ge (t + T)(t - 1)$. Second, we have the following asymptotic expansion, which will be used in our analysis.

\begin{proposition}
\label{proposition: phi representation}
Fix some $q$, $t \le q \le \min\{T, k\}$. Let $\cS$ be a family satisfying~\ref{item: sunflower freeness S*}-\ref{item: lower levels emptiness of S*} with $\cS$ in place of $\cS_*$. Then, we have
\begin{align*}
    |\ff_\cS| = \sum_{i = 0}^{q - t} \widetilde{\phi}_i(\cS) \binom{n - T}{k - t - i} + O_{s,k}(n^{- 1}) \binom{n}{k - q}.
\end{align*}
\end{proposition}

We would like to mention that the reduction to the problem of maximizing $\widetilde{\phi}(\cdot)$ subject to a property such as avoiding a sunflower is a general phenomenon. In our upcoming paper~\cite{kupavskii_noskov_followup}, we will show that the same reduction can be obtained for a large class of Turan-type problems.

\section{Tools}
\label{section: tools}

We start this section by describing the isoperimetric inequality for expander graphs. Given a $d$-regular graph $G = (V, E)$ and a subset of its vertices $S$, we denote by $\partial S$ the set of edges $e$ such that $|e \cap S| = 1$, i.e., the edges $e$ that connect vertices of $S$ with vertices of $V \setminus S$. Define the Laplacian $L_G$ of the graph $G$ as $L_G = d \cdot I_{|V|} - A_G$, where $A_G$ stands for the adjacency matrix of $G$. Then, the following holds:

\begin{lemma}[Theorem 20.1.1 from~\cite{spielman2019spectral}]
\label{lemma: Cheeger inequality}
Let $\lambda_2(L_G)$ be the second smallest eigenvalue of $L_G$. Then, for any $S$ of size at most $|V|/2$, we have
\begin{align*}
    |\partial S| \ge \lambda_2(L_G) |S| /2.
\end{align*}
\end{lemma}

To prove Theorem~\ref{theorem: graph case Duke-Erdos extremal}, we need a characterisation of extremal examples for Problem~\ref{problem: erdos sunflower problem}. The following theorem was established by Abbot et al.~\cite{abbott1972intersection}. Although they did not explicitly state the uniqueness of extremal examples, it can be easily deduced from their proof.
\begin{theorem}
\label{theorem: sunflower graph case}
Let $G$ be a graph without a matching of size $s$ and a vertex of degree $s$. Then, 
\begin{align*}
    |E(G)| \le \begin{cases}
        s (s - 1), & \text{if } s \text{ is odd}, \\
        (s - 1)^2 + (s - 2)/2, & \text{if } s \text{ is even}.
    \end{cases}
\end{align*}
Assuming that $G$ does not contain isolated vertices, the graphs that attain the bound  are as follows. If $s$ is odd,  equality is attained if and only if $G$ is a disjoint union of two cliques of size $s$. If $s$ is even and $s > 2$,  equality is attained if and only if $G$ is a graph on $2s - 1$ vertices with degrees $d_1 = (s - 2), d_2 = \ldots = d_{2s - 1} = (s - 1)$. If $s = 2$,  equality is attained if and only if $G$ consists of a single edge.
\end{theorem}

Finally, we need the Kruskal--Katona theorem in the Lov\'asz form.

\begin{theorem}[{\cite[Proposition 3.4]{frankl1987shifting}}]
\label{theorem: kruskal--katona theorem}
Let $\ff$ be a subfamily of $\binom{[n]}{k}$ such that $|\ff| = \binom{x}{k}$ for some real positive $x$. Then we have $|\partial_h \ff| \ge \binom{x}{h}$.
\end{theorem}

\section{Proof of Theorem~\ref{theorem: stability in Erdos-Duke}}
\label{section: stability of ED -- proof}

Following Frankl and F\"uredi, we use the delta-system method. The original approach of Frankl and F\"uredi heavily relies on the following lemma (see Lemma 7.4 from~\cite{Frankl1987}).

\begin{lemma}
\label{lemma: delta-system method}
    Let $\ff \subset \binom{n}{k}$. Fix numbers $p, t$ such that $p \ge k \ge 2 t + 1$. Suppose that $\ff$ does not contain a sunflower with $p$ petals and the core of size $t - 1$. Then there exists a subfamily $\ff^* \subset \ff$ with the following properties:
    \begin{enumerate}
        \item for any $F \in \ff^*$, there exists a subset $T(F)$ of size $t$ such that each set $E$, $T \subset E \subsetneq F$, is a core of a sunflower with $p$ petals in $\ff^*$;
        \item $|\ff \setminus \ff^*| \le c(p, k) \binom{n}{k - t - 1}$,
    \end{enumerate}
    where $c(p, k)$ is some function that can be bounded by $p^{2^k} \cdot 2^{2^{Ck}}$ for some absolute constant $C$. 
\end{lemma}

For the sake of completeness, in this section we present  the proof of the original result by Frankl and F\"uredi on the Duke--Erd\H os problem. The core idea of their proof will be used in the proof of our results. 

\begin{theorem}\cite{Frankl1987}
\label{theorem: delta-system solution}
    Let $k \ge 2t + 1$. Suppose that a family $\ff\subset{[n]\choose k}$ does not contain a sunflower with $s$ petals and the core of size $t - 1$. Then
    \begin{align*}
        |\ff| \le \phi(s, t) \binom{n}{k - t} + O_{s,k} \left ( \frac{1}{n} \right ) \binom{n}{k - t}.
    \end{align*}
\end{theorem}

\begin{proof}
Applying Lemma~\ref{lemma: delta-system method} with $p = sk$, we infer that there exists $\ff^* \subset \ff$, $|\ff \setminus \ff^*| \le c(sk, k) \binom{n}{k - t - 1}$, such that for any $F \in \ff^*$ there is a subset $T \subset F$ of size $t$ such that any $E$, $T \subset E \subsetneq F$, is a core of a sunflower with $sk$ petals in $\ff^*$. For each $F \in \ff^*$, fix one such $T = T(F)$. Then, we have
\begin{align*}
    |\ff^*| = \sum_{D \in \binom{[n]}{k - t}} |\{F \in \ff^* \mid F \setminus D = T(F) \}|.
\end{align*}
Define $\ff_D = \{F\setminus D\mid F  \in \ff^*, F \setminus D = T(F) \}$. We claim that $\ff_D$ does not contain a sunflower with $s$ petals and arbitrary core. Otherwise, take such a sunflower $T_1, \ldots, T_s$. Choose $X \subset D$, $|X| = t - 1 - |\cap_i T_i|$. Each $T_i \cup X$ is the core of a sunflower in $\ff^*$ with $sk$ petals. Then the following claim applied to $(T_1 \cup X), \ldots, (T_s \cup X)$ leads to a contradiction.

\begin{claim}
\label{claim: no sk cores of sunflower}
Suppose that $X_1, \ldots, X_s$ form a sunflower, moreover, each $X_i$ is a core of a sunflower with $sk$ petals in $\ff^*$. Then there exist pairwise disjoint sets $F_i \in \ff^*(X_i), i = 1, \ldots, s$, that do not intersect $\bigcup_i X_i$. Consequently, $\ff^*$ contains a sunflower $F_1 \sqcup X_1, \ldots, F_s \sqcup X_s$ with the core $\cap X_i$.
\end{claim}
\begin{proof}
    Let $m$ be the maximum number such that there are disjoint $F_1 \in \ff^*(X_1), \ldots, F_m \in \ff^*(X_m)$ that do not intersect $\cup_{i = 1}^m X_i$. Since $\ff^*$ does not contain a sunflower with $s$ petals and the core of size $t - 1$, $m$ is strictly less than $s$. Consider a family $\ff^*(X_{m + 1})$. By assumption, it contains pairwise disjoint sets $G_1, \ldots, G_{sk}$. At least one of these sets does not intersect $\cup_{i = 1}^m (X_i \cup F_i)$ since $|\cup_{i = 1}^m (T_i \cup F_i)| \le (s - 1) k < sk$. Hence, we can choose $F_{m + 1} \in \ff^*(X_{m+1})$ that does not intersect $\cup_{i = 1}^{m + 1} X_i$, contradicting the maximality of $m$. 
\end{proof}

Therefore, $\ff_D \subset \binom{n}{t}$ does not contain a sunflower with $s$ petals, and so $|\ff_D| \le \phi(s, t)$. Since 
\begin{align}
    |\ff^*| \le \sum_{D \in \binom{[n]}{k - t}} |\ff_D|, \label{eq: bound via shadow}
\end{align}
we have
\begin{align*}
    |\ff| \le \phi(s,t) \binom{n}{k - t} + c(sk, k) \binom{n}{k - t - 1} \le \phi(s,t) \binom{n}{k - t} + \frac{k \cdot c(sk, k)}{n - k} \binom{n}{k - t}. & \qedhere
\end{align*}
\end{proof}

We are ready to prove Theorem~\ref{theorem: stability in Erdos-Duke}.

\begin{proof}[Proof of Theorem~\ref{theorem: stability in Erdos-Duke}]
Consider a family $\ff$ from the statement of the theorem. We have
\begin{align*}
    |\ff| \ge (1 - \delta) \phi(s, t) \binom{n}{k - t}.
\end{align*}
Let $\ff^*$ and $\{\ff_{D}\}_{D \in \binom{[n]}{k - t}}$ be the families obtained in the proof of Theorem~\ref{theorem: delta-system solution}. Since $|\ff \setminus \ff^*| \le O_{s,k} \left (n^{-1} \right ) \binom{n}{k - t}$, the bound~\eqref{eq: bound via shadow} implies
\begin{align*}
    \sum_{D \in \binom{[n]}{k - t}} |\ff_D| \ge |\ff^*| \ge |\ff| -  O_{s,k} (1/n) \cdot \binom{n}{k - t} \ge \left (1 - \delta - O_{s,k}(1/n) \right ) \phi(s, t)\binom{n}{k - t}.
\end{align*}
Define
\begin{align*}
    \cD = \left \{D \in \binom{[n]}{k - t} \mid |\ff_D| = \phi(s, t) \right \}.
\end{align*}
Then, we have
\begin{align*}
    \sum_{D \in \cD} |\ff_D| + \sum_{D \in \binom{[n]}{k - t} \setminus \cD} |\ff_D| \le \phi(s, t) |\cD| + (\phi(s, t)  -  1) \left (\binom{n}{k - t} - |\cD| \right ).
\end{align*}
It implies
\begin{align}
    |\cD| & \ge \left (1 - \delta - O_{s,k}(1/n) \right ) \phi(s, t) \binom{n}{k - t} - (\phi(s, t) - 1) \binom{n}{k - t} \nonumber \\
    & \ge (1 - \delta \cdot \phi(s,t) - O_{s,k}(1/n)) \binom{n}{k - t}. \label{eq: saturated families lower bound}
\end{align}

Recall the definition of the Johnson graph $J_{n, k - t}$. Its vertex set is  $\binom{[n]}{k - t}$, and two sets $D_1, D_2$ are connected by an edge iff $|D_1 \cap D_2| = k - t- 1$. We claim that if $D_1, D_2 \in \cD$ are adjacent in $J(n, k - t)$, then $\ff_{D_1}(D_1) = \ff_{D_2}(D_2)$. Indeed, if they are not equal, then $|\ff_{D_1}(D_1) \cup \ff_{D_2}(D_2)| \ge \min_{i = 1, 2}{|\ff_{D_i}|} + 1 \ge \phi(s, t) + 1$, where we used that $D_1, D_2 \in \cD$. Hence, $\ff_{D_1}(D_1) \cup \ff_{D_2}(D_2)$ contains a sunflower $T_1, \ldots, T_s$. Let $C = \cap_{i = 1}^s T_i$, $|C| \le t - 1$. Since $k\ge 2t+1$, we have
\begin{align*}
    |D_1 \cap D_2|  = k - t - 1\ge t - 1.
\end{align*}
Thus, there exists a set $M \subset D_1 \cap D_2$ of size $t - 1  - |C|$. By the construction of $\ff^*$, each $T_1 \cup M, \ldots, T_s \cup M$ is a center of a sunflower with $sk$ petals and the core of size exactly $t - 1$. By Claim~\ref{claim: no sk cores of sunflower}, it implies that $\ff^*$  contains a sunflower $F_1, \ldots, F_s$, $T_i \cup M \subset F_i$ with $s$ petals and the core of size $t - 1$, a contradiction.

It means that for any connected component $\cS$ in the subgraph $J_{n, k - t}[\cD]$ of $J_{n, k - t}$ induced on $\cD$ there exists a family $\cT$ such that $\ff_D = \cT$ for all $D \in \cS$. Using a simple isoperimetric argument, we will prove that the induced subgraph of $\cD$ is almost fully covered by some giant connected component.  

First, let us describe the spectrum of the adjacency matrix of $J_{n, k - t}$. The eigenvalues of $J_{n, k - t}$ are $\{(k - t - j)(n - k + t- j) - j\}_{j = 0}^{\diam{J(n, k - t)}}$, see Section 14.4.2 of~\cite{brouwer2011spectra}. Since $n \ge 2 (k - t)$, the diameter of the graph is  $k - t$. The spectral gap, that is, the difference between the largest and the second largest eigenvalue of the adjacency matrix, is
\begin{align}
    & (k - t) (n - k + t) - (k - t - 1)(n - k + t - 1) + 1 = \nonumber \\
    & \quad (k - t) (n - k + t) - (k - t) (n - k + t) + (k - t ) + (n - k + t) - 1 + 1 = \nonumber \\
    & \quad (k - t) + (n - k + t) = n. \label{eq: Johnson graph spectral gap}
\end{align}
Since the Johnson graph is regular, the spectral gap of the adjacency matrix is equal to the second eigenvalue of the Laplacian.

Next, let $\bbC[\cD]$ be the set of connected components of $J_{n, k - t}[\cD]$. Let $\cS_0$ be the largest connected component of $J_{n, k - t}[\cD]$. Then, for any other connected component $\cS \in \bbC[\cD] \setminus \{\cS_0\}$, we have the following two properties.
\begin{enumerate}
    \item The size of $\cS$ does not exceed $\frac 12{n\choose k-t}$. Otherwise, $|\cS| + |\cS_0| \ge 2 |\cS| > \binom{n}{k - t}$, a contradiction.
    \item The edges of $\partial \cS$  in $J_{n, k - t}$ are incident to vertices in  $\binom{[n]}{k - t} \setminus \cD$.
\end{enumerate}
Applying Lemma~\ref{lemma: Cheeger inequality} and denoting the maximal degree of $J_{n, k - t}$ by $\Delta(J_{n, k - t})$, we get
\begin{align*}
    \frac{n}{2} |\cD \setminus \cS_0| = \frac{n}{2} \sum_{\cS \in \bbC[\cD] \setminus \{\cS_0\}} |\cS| \le \sum_{\cS \in \bbC[\cS] \setminus \{\cS_0\}} |\partial \cS| \le \Delta(J_{n, k - t}) \cdot \left |\binom{[n]}{k - t} \setminus \cD \right |.
\end{align*}
Using $\Delta(J_{n, k - t}) = (k - t) \cdot (n - k + t)$ and~\eqref{eq: saturated families lower bound}, we get
\begin{align*}
    |\cD \setminus \cS_0| \le 2 k \cdot \left (\delta \cdot \phi(s, t) + O_{s,k}(1/n) \right ) \cdot \binom{n}{k - t}.
\end{align*}

Let $\cT_0 \subset \binom{[n]}{t}$ be the family such that $\ff_D(D) = \cT_0$ for all $D \in \cS_0$. Then, we have
\begin{align*}
    |\ff \setminus \ff[\cT_0]| & \le \sum_{D \in \binom{[n]}{k - t}  \setminus \cS_0} |\ff_D| \le \phi(s, t) \cdot \left | \binom{n}{k - t}  \setminus \cS_0 \right | \\
    & \le \phi(s, t) \cdot \left ( \left | \binom{n}{k - t} \setminus \cD \right | + |\cD \setminus \cS_0| \right ) \\
    & \le \left ( 3k \phi^2(s, t) \delta + O_{s,k}(1/n) \right ) \binom{n}{k - t} = \left ( 3k \phi^2(s, t) \delta + O_{s,k}(1/n) \right ) \binom{n - t}{k - t}
\end{align*}
for sufficiently large $n$.
\end{proof}

\section{Proof of Theorem~\ref{theorem: structural duke-erdos}}
\label{section: proof of structural duke-erdos}

We start with the following corollary of Theorem~\ref{theorem: stability in Erdos-Duke} and Example~\ref{example: basic example}.

\begin{corollary}
\label{corollary: corollary of the stability}
Let $k \ge 2t + 1$. Let $\ff$ be an extremal family that does not contain a sunflower with $s$ petals and the core of size $t - 1$. Then, there exists a family $\cT \subset \binom{[n]}{t}$ that does not contain a sunflower with $s$ petals such that $|\cT| = \phi(s, t)$ and
\begin{align*}
    |\ff \setminus \ff[\cT]| \le \frac{C(s, k)}{n} \binom{n - t}{k - t}
\end{align*}
for some constant $C(s, k)$ depending on $s, k$ only and $n$ sufficiently large as a function of $s, k$.
\end{corollary}

\begin{proof}
    Let $\ff_1$ be a family from Example~\ref{example: basic example} for some $\cT_1$, $|\cT_1| = \phi(s, t)$. Then, we have
    \begin{align*}
        |\ff_1| & \ge \phi(s, t) \binom{n - |\support \cT_1|}{k - t} \ge \phi(s, t) \left (\binom{n}{k - t} - |\support \cT_1| \binom{n - 1}{k - t - 1} \right )  \\
        & \ge \phi(s, t) \left ( 1 - \frac{k t \phi(s, t)}{n} \right ) \binom{n}{k - t},
    \end{align*}
    where we used $|\support \cT_1| \le t \phi(s,t )$. Hence, an extremal family $\ff$ satisfies
    \begin{align*}
        |\ff| \ge |\ff_1|\ge (1 - \delta) \phi(s, t) \binom{n}{k - t}
    \end{align*}
    for $\delta = tk\phi(s, t)/n$. Due to Theorem~\ref{theorem: stability in Erdos-Duke}, there exists a family $\cT \subset \binom{[n]}{t}$ that does not contain a sunflower with $s$ petals such that
    \begin{align}
       \label{eqboundonrem} |\ff \setminus \ff[\cT]| \le \left ( \frac{tk^2 \phi(s, t)^3}{n} + \frac{C(s, k)}{n} \right ) \binom{n - t}{k - t} \le \frac{C'}{n} \binom{n - t}{k - t}
    \end{align}
    for some constant $C'=C'(s, k)$. Since $|\ff| \ge (1 - O_{s,k}(1/n)) \phi(s, t) \binom{n - t}{k - t}$, we have $|\cT| = \phi(s, t)$.
\end{proof}

To prove Theorem~\ref{theorem: structural duke-erdos}, we use the following claim.

\begin{claim}
\label{claim: each edge is fat}
Let $k \ge 2t + 1$. Let $\ff$ be an extremal family that does not contain a sunflower with $s$ petals and the core of size $t - 1$. Let $\cT$ be the family from Corollary~\ref{corollary: corollary of the stability}. Then, uniformly over $T \in \cT$, we have
\begin{align*}
    |\ff(T)| = \left ( 1- O_{s,k} (1/n) \right ) \binom{n - t}{k - t}.
\end{align*}
\end{claim}

\begin{proof}
 Fix some $T \in \cT$. Then, we have
    \begin{align*}
        \phi(s, t) \left (1 - \frac{t k \phi(s, t)}{n} \right ) \binom{n - t}{k - t} \le |\ff| & \overset{\eqref{eqboundonrem}}{\le} |\ff[\cT]| + \frac{C'}{n} \binom{n - t}{k - t} \\
        & \le (|\cT| - 1) \binom{n - t}{k - t} + |\ff(T)| + \frac{C'}{n} \binom{n - t}{k - t} \\
        & \le (\phi(s, t) - 1) \binom{n - t}{k - t} + |\ff(T)| + \frac{C'}{n} \binom{n - t}{k - t}.
    \end{align*}
    Hence, we have
    \begin{align*}
        |\ff(T)| \ge \left ( 1 - \frac{tk \phi^2(s,t)}{n} - \frac{C'}{n} \right ) \binom{n - t}{k - t} = \left (1 - O_{s,k}(1/n) \right ) \binom{n -t }{k - t}.
    \end{align*}
    Also, trivially, $|\ff(T)| \le \binom{n - t}{k - t}$. The claim follows.
\end{proof}

The following lemma is the key technical ingredient in the proof of our main results. 

\begin{lemma}
\label{lemma: induction over layers}
Take $k \ge 2t + 1$. Let $\ff \subset \binom{[n]}{k}$ be an extremal family without a sunflower with $s$ petals and the core of size $t - 1$.  Fix $q \ge t$ such that $k \ge (q + 1) (t - 1)$. Then, there exists a family $\cS \subset \bigcup_{j = t}^q \binom{[n]}{j}$ with the following properties:
\begin{enumerate}
    \item \label{item: induction sunflower restriction} $\cS$ does not contain a sunflower with $s$ petals and the core of size at most $t - 1$;
    \item \label{item: fat restrictions over S} for any $S \in \cS$, we have
    \begin{align*}
        |\ff(S, \support \cS^{(t)})| = (1 - O_{s,k}(1/n)) \binom{n}{k - |S|};
    \end{align*}
    \item \label{item: induction support restriction} $\cS \subset 2^{\support \cS^{(t)}}$;
     \item \label{item: induction max lexicographical order} the vector $(\widetilde{\phi}_0(\cS), \ldots, \widetilde{\phi}_{\min\{q , T\} - t}(\cS))$ is maximal in the lexicographical order over all choices of $\cS$ satisfying~\ref{item: induction sunflower restriction},\ref{item: induction support restriction};
\item \label{item: remainder bound} For the family
\begin{align*}
    \ff_{\cS} = \left \{ F \in \binom{[n]}{k} \mid F \cap \support \cS^{(t)} \in \cS \right \}
\end{align*}
we have
\begin{align}
\label{eq: q-th remainder bound}
    |\ff \setminus \ff_{\cS}| = O_{s,k}(n^{- (q - t) - 1}) \cdot \binom{n - t}{k - t}.
\end{align}
Moreover, the family
\begin{align*}
    \cR =\{F \in \ff \mid |F \cap \support \cS^{(t)}| \le q \text{ and } F \cap \support \cS^{(t)} \not \in \cS \}
\end{align*}
is a subfamily of 
\begin{align*}
    \widetilde{\cR} = \{F \in \ff \mid & \exists S_1, \ldots, S_{s - 1} \in \cS \text{ such that}  \\
    & \text{$F, S_1, \ldots, S_{s-1}$ form}  \text{ a sunflower with $s$ petals} \\
    & \text{and the core of size at most $t - 1$} \},
\end{align*}
which has size $O_{s,k} \left ( n^{- \frac{k - t (t - 1)}{t - 1}} \right ) \binom{n - t}{k - t}$.
\end{enumerate}
\end{lemma}
We note that the support of $\cS$ coincides with the support of $\cS^{(t)}$ and has size at most $T$ and is in particular independent of $n$.
\begin{proof} For clarity, we divide the proof into several steps.

    \noindent \textbf{Step 1. Description of the proof structure.} We construct the family $\cS$ iteratively in $q - t + 1$ steps. At step $i$, $i = 0, \ldots, q - t$, we will construct a family $\cS_i$ of uniformity at most $t + i$ such that it satisfies properties~\ref{item: induction sunflower restriction}-\ref{item: induction max lexicographical order} (with $\cS_i$ in place of $\cS$ and $t + i$ in place of $q$) and such that
    \begin{align}
    \label{eq: definition of R i}
        \cR_i = \{F \in \ff \mid |F \cap \support \cT| \le t + i \text{ and } F \cap \support \cS_i^{(t)} \not \in \cS_i \}
    \end{align}
    has size $O_{s,k} \left (n^{- \frac{k - t (t - 1)}{t - 1}} \right ) \binom{n - t}{k - t}$. Then, we put $\cS = \cS_{r}$, where $r = \min\{|\support \cS_0| - t, q - t\}$. For brevity, we denote $\support \cS_0$ by $\bfS$.

    \noindent \textbf{Step 2. Construction of $\cS_0$.} Let us construct $\cS_0$. Due to Corollary~\ref{corollary: corollary of the stability}, there exists a set $\cT$, $|\cT| = \phi(s, t)$, such that $|\ff \setminus \ff[\cT]| = O_{s,k}(1/n) \binom{n - t}{k - t}$. Set $\cS_0 = \cT$. Property~\ref{item: induction sunflower restriction} follows from the definition of $\cT$. Property~\ref{item: fat restrictions over S} is implied by Claim~\ref{claim: each edge is fat}. Property~\ref{item: induction support restriction} is trivial. Clearly, the vector $(|\cT|)$ is maximal in the lexicographical order over all choices of $\cT$ satisfying properties~\ref{item: induction sunflower restriction}-\ref{item: induction support restriction} for $q = t$, so property~\ref{item: induction max lexicographical order} holds. Then, we bound $\cR_0$ by the following proposition. Its proof relies on the Kruskal--Katona theorem and is postponed to Section~\ref{subsection: proof of kruskal-katona type bound for remainder}.

    \begin{proposition}
    \label{proposition: bounds on same uniformity remainder}
    Suppose $\cS_i$ satisfies properties~\ref{item: induction sunflower restriction}-\ref{item: induction support restriction} with $\cS_i$ in place of $\cS$. Define
    \begin{align*}
        \widetilde{\cR}_i = \{F \in \ff \mid & \exists S_1, \ldots, S_{s - 1} \in \cS_i \text{ such that}  \\
    & \text{$F, S_1, \ldots, S_{s-1}$ form}  \text{ a sunflower with $s$ petals} \\
    & \text{and the core of size at most $t - 1$} \}.
    \end{align*}
    Then, we have
    \begin{align*}
        |\widetilde{\cR}_i| = O_{s,k} \left (n^{- \frac{k - t (t - 1)}{t - 1}} \right ) \binom{n - t}{k - t}.
    \end{align*}
    \end{proposition}

    By the maximality of $\cT$ (i.e., $|\cT| = \phi(s,t)$) and the definition~\eqref{eq: definition of R i} of $\cR_0$, for any $F \in \cR_0$ there exist sets $T_1, \ldots, T_{s-1} \in \cT$ such that $F, T_1, \ldots, T_s$ form a sunflower with $s$ petals and the core of size at most $t - 1$ (otherwise, we can choose any $X \in \binom{F \setminus \support \cT}{t - |F \cap \support \cT|}$ and add $X \cup (F \cap \support \cT)$ to $\cT$). Hence, $\cR_0 \subset \widetilde{\cR}_0$ and, by Proposition~\ref{proposition: bounds on same uniformity remainder}, we have $|\cR_0| = O_{s,k} \left (n^{- \frac{k - t (t - 1)}{t - 1}} \right ) \binom{n - t}{k - t}$.

    \noindent \textbf{Step 3. $\cS_{i-1}\to \cS_i$.} In what follows, we  describe how to construct $\cS_i$, given $\cS_{i - 1}$. We assume that $\cS_{i - 1}$ satisfies properties~\ref{item: induction sunflower restriction}-\ref{item: induction max lexicographical order} of the lemma, and $\cR_{i - 1} \subset \widetilde{\cR}_{i - 1}, \; |\widetilde{\cR}_{i - 1}| = O_{s,k} \left (n^{- \frac{k - t (t - 1)}{t - 1}} \right ) \binom{n}{k - t}$.

    We consider two cases. If $t + i - 1 = |\bfS|$, then put $\cS = \cS_{i - 1}$, $\cR = \cR_{i - 1}$, and stop. Clearly, we have $\ff \setminus \ff_{\cS} \subset \cR$. Since $k \ge (q + 1)(t - 1)$ and $\cR = O_{s,k} \left ( n^{-\frac{k - t (t - 1)}{(t - 1)} }\right ) \binom{n}{k - t}$, we have $|\ff \setminus \ff_S| = O_{s,k}(n^{-(q - t) - 1}) \binom{n - t}{k - t}$, and the lemma follows. Hence, we assume that $t + i \le |\bfS| \le T$.

    \noindent \textbf{Step 4. Decomposing $\ff \setminus \ff_{\cS_{i - 1}}$.} Define the following families:
    \begin{align}
        \ff_i & = \{F \in \ff \mid |F \cap \bfS| = t + i \}, \nonumber \\
        \widetilde{\cS}_i & = \left \{S \in \binom{|\bfS|}{t + i} \mid |\ff_i(S)| \ge sk \binom{n - |\bfS|}{k - t - i - 1} \right \} \label{eq: tilde S_i definition}\\
        \cD_{i,S}^* & = \bigg \{ D \in \binom{[n] \setminus \bfS}{k - t - i} \mid \forall p \in [t - 1] \; \forall X \in \binom{D}{p}  \nonumber \\
        & \qquad  \quad |\ff_i(S \sqcup X)| \ge sk \binom{n - |\bfS| - |X|}{k - t- i - |X|- 1}\bigg \}, \quad S \in \widetilde{\cS}_i, \label{eq: D stap iS definition}\\
        \ff_i^* & = \bigcup_{S \in \widetilde{\cS}_i} \left ( \{S\} \vee \{\cap_{S \in \widetilde{\cS}_i} \cD^*_{i, S}\} \right ) \label{eq: F star i definition}.
    \end{align}
    Informally, we include in $\ff_i^*$ only those sets $F\in \ff_i$ for which any subset of $F$ containing their intersection with $\bfS$ and having at most $t-1$ elements outside $\bfS$ has high degree in $\ff_i$. This yields the following partition of $\ff$:
    \begin{align*}
        \ff = \ff_{\cS_{i - 1}} \sqcup \ff^*_i \sqcup (\ff_i \setminus \ff^*_{i}) \sqcup \cO_i \sqcup \cR_{i - 1},
    \end{align*}
    where $\cO_i$ is defined as
    \begin{align*}
        \cO_i = \{F \in \ff \mid |F \cap \bfS| \ge t + i + 1\}.
    \end{align*}
    Clearly, we have $|\cO_i| \le 2^{|\bfS|} \binom{n}{k - t - i - 1} \le 2^{t \phi(s, t)} \binom{n}{k - t - i - 1} = O_{s,k} (n^{-i - 1}) \binom{n - t}{k - t}$.
    Let us bound $|\ff_i \setminus \ff^*_i|$. For each set $F \in \ff_i \setminus \ff^*_i$, there are two options:
    \begin{enumerate}
        \item $F \cap \bfS \not \in \widetilde{\cS}_i$;
        \item $F \cap \bfS \in \widetilde{\cS}_i$ and $F \setminus \bfS \not \in \cD^*_{i, S}$ for some $S \in \widetilde{\cS}_i$, i.e., there exist a number $p \in [t - 1]$ and a set $X \subset F \setminus \bfS$ such that $\ff_i(S \sqcup X) \le sk \binom{n - |\bfS| -  |X|}{k - t- i - |X|- 1}$.
    \end{enumerate} 
    For each $S \in \widetilde{\cS}_i$, define
    \begin{align*}
        \cE_{i, p, S} = \bigg \{X \in \binom{[n] \setminus\bfS}{p} \mid \ff_i(S \sqcup X) < sk \binom{n - |\bfS| - |X|}{k - t- i - |X|- 1} \bigg \}.
    \end{align*}
    Then, we have
    \begin{align*}
        |\ff_i \setminus \ff_i^*| & \le \sum_{S \in \binom{\bfS}{t + i} \setminus \widetilde{\cS}_i} |\ff_i[S]| +  \sum_{S \in \widetilde{\cS}_i} \sum_{p = 1}^{t - 1} |\ff[\cE_{i,p, S}]| \\
        & \le sk \binom{|\bfS|}{t + i} \binom{n - |\bfS|}{k - t- i - 1} + sk \sum_{S \in \binom{\bfS}{t + i}} \sum_{p = 1}^{t - 1} |\cE_{i,p, S}| \binom{n - |\bfS| - p}{k - t- i - p - 1} \\
        & \le O_{s,k}(n^{-i - 1}) \binom{n - t}{k - t} +  sk \binom{|\bfS|}{t + i} \sum_{p = 1}^{t - 1} \binom{n}{p} \binom{n}{k - t - i - p - 1} \\
        & = O_{s,k} \left ( n^{- i - 1} \right ) \binom{n - t}{k - t}.
    \end{align*}
    This yields the following bound on $|\ff \setminus (\ff_{\cS_{i - 1}} \sqcup \ff^*_i)|$:
    \begin{align}
        |\ff \setminus (\ff_{\cS_{i - 1}} \sqcup \ff^*_i)| & \le |\ff_i \setminus \ff_i^*| + |\cO_i| + |\cR_{i - 1}| \nonumber \\
        & = O_{s,k}(n^{- i - 1}) \binom{n - t}{k - t} + O_{s,k}(n^{- i - 1}) \binom{n - t}{k - t} + O_{s,k} \left ( n^{-\frac{k - t (t - 1)}{t - 1}} \right ) \binom{n - t}{k - t} \nonumber \\
        &= O_{s,k} \left ( n^{- i - 1} \right ) \binom{n - t}{k - t}, \label{eq: removing raminder on i-th layer}
    \end{align}
    where we used $k \ge (q + 1) (t - 1)$ in the third equality.
    
    \noindent \textbf{Step 5. Upper bound on $|\ff^*_i|$.} In order to analyze the structure in $\ff_i^*$, we use the decomposition
    \begin{align*}
        |\ff^*_i| = \sum_{D \in \binom{[n] \setminus \bfS}{k - t- i}} |\ff^*_i(D)|.
    \end{align*}
    Our ultimate goal in that respect is to show that almost all $\ff_i^*(D)$ are, in fact, the same. This common family will  then be added to $\cS_{i-1}$ in order to form $\cS_i^{(t+i)}$.
    
    Define
    \begin{align}
        \cD_{i - 1} = & \bigg \{ D \in \binom{[n] \setminus \bfS}{k - t- i} \mid  \forall S \in \cS_{i - 1} \; \forall p \in [t - 1] \; \forall X \in \binom{D}{p} \nonumber \\ 
        &\qquad \qquad |\ff(S \sqcup X, \bfS \sqcup X)| \ge sk \binom{n - |\bfS|- p}{k - p - |S| - 1}   \bigg \}. \label{eq: Di-1 definition}
    \end{align}
    By property~\ref{item: fat restrictions over S} we have $|\ff(S)| = (1 - O_{s,k}(1/n)) \binom{n}{k - |S|}$ for each $S \in \cS_{i - 1}$. This allows us to bound $|\cD_{i - 1}|$ from below. We state it as a separate  proposition. Its proof is postponed to Section~\ref{subsection: proof of k-t-i fatness}.
    \begin{proposition}
    \label{proposition: k-t-i fatness}
    We have
    \begin{align*}
        \left |\binom{[n] \setminus \bfS}{k - t - i}  \setminus \cD_{i - 1} \right | = O_{s,k}(1/n) \binom{n- |\bfS|}{k - t- i}
    \end{align*}
    \end{proposition} 
    The proposition implies
    \begin{align}
        |\ff^*_i| & = \sum_{D \in \binom{[n] \setminus \bfS}{k - t- i}} |\ff^*_i(D)| \nonumber \\
        & \le \sum_{D \in \cD_{i - 1}} |\ff^*(D)| + \binom{|\bfS|}{t + i} \cdot \left |\binom{[n] \setminus |\bfS|}{k - t- i}  \setminus \cD_{i - 1} \right | \nonumber \\
        & = \sum_{D \in \cD_{i - 1}} |\ff^*(D)| + O_{s,k} (n^{-i - 1}) \binom{n - t}{k - t} \label{eq: bound with fat D i of S i}.
    \end{align}
    To bound $\ff^*_i(D)$ for each $D \in \cD_{i - 1}$, we introduce the following quantity:
    \begin{align}
        \phi_i = \max \bigg \{|\cQ| \mid & \cQ \subset \binom{\bfS}{t + i} \text{ and } \label{eq: definition of phi i} \\
        & \cS_{i - 1} \cup \cQ \text{ does not contain a sunflower with $s$ petals} \nonumber \\
        & \text{and the core of size at most $t - 1$} \bigg \} . \nonumber
    \end{align}
    We claim that $|\ff^*_i(D)| \le \phi_i$. It is enough to show that $\cS_{i - 1} \sqcup \ff^*_i(D)$ does not contain a sunflower with $s$ petals and the core of size at most $t - 1$. Let $T_1, \ldots, T_{s}$ be such a sunflower. Define $\ell = t - 1- |\cap_i T_i|$. WLOG, we assume that $T_1, \ldots, T_p \in \cS_{i - 1}$ and $T_{p + 1}, \ldots, T_s \in \ff^*_i(D)$.   Note that $T_j \in \widetilde{\cS}_i$ for $j > p$ by the definition of $\ff^*_i$, so $\ff^*_i(T_p) \ge sk \binom{n - |\bfS|}{k - t- i - 1}$ by the definition~\eqref{eq: tilde S_i definition} of $\widetilde{\cS}_i$. Consider two cases.

    \textbf{Case 1. We have $\ell = 0$.} We rely on the following standard proposition, which proof is postponed to Section~\ref{subsection: proof of cross-matching}.

    \begin{proposition}
    \label{proposition: cross matching}
    Consider the families $\cA_1 \subset \binom{U}{k_i}, \cA_2 \subset \binom{U}{k_2}, \ldots, \cA_s \subset \binom{U}{k_s}$ such that $|\cA_i| \ge \left (\sum_{i = 1}^s k_i \right ) \binom{|U|}{k_i - 1}$. Then, there exist disjoint $F_1 \in \cA_1, \ldots, F_s \in \cA_s$.
    \end{proposition}

    The constructed family $\cS_{i - 1}$ satisfies property~\ref{item: fat restrictions over S}, so $|\ff(T_j, \bfS)| \ge sk \binom{n - |\bfS|}{k - |T_p| - 1}$ for $j \in [p]$, provided $n$ is large enough. Applying Proposition~\ref{proposition: cross matching} with $U=[n] \setminus \bfS$  and $k_i = k - |T_i|$, we deduce that there are disjoint sets $F_1 \in \ff(T_1, \bfS), \ldots, F_s \in \ff^*_i(T_s) $. Note that $\ff^*_i(T_s) = \ff^*_i(T_s, \bfS)$ by the definition~\eqref{eq: tilde S_i definition} of $\widetilde{\cS}_i \ni T_s$ and the construction~\eqref{eq: F star i definition} of $\ff^*_i$. Hence, $\ff$ contains a sunflower $T_1 \sqcup F_1, \ldots, T_s \sqcup F_s$ with $s$ petals and the core of size $t - 1$.

    \textbf{Case 2. We have $\ell > 0$.} Then, choose any $X \in \binom{D}{\ell}$. Note that for any $S \in \ff^*(D)$, we have $|\ff^*(S \sqcup X)| \ge \binom{n - |\bfS| - |X|}{k - t - i - |X| - 1}$ by the definition~\eqref{eq: D stap iS definition} of $\cD_{i, S}^*$. In particular, $|\ff^*_i(T_j \sqcup X)| \ge sk \binom{n - |\bfS| - |X|}{k - t - i - |X| - 1}$ for $j = p+1, \ldots, s$. Since $D \in \cD_{i - 1}$, we have $|\ff(T_j \sqcup X, \bfS \sqcup X)| \ge sk \binom{n - |\bfS| - |X| }{k - t- i - |X| - 1}$ for all $j \in [p]$. Then, we apply Proposition~\ref{proposition: cross matching} again, and find disjoint $F_1 \in \ff(T_1 \sqcup X, \bfS  \sqcup X), \ldots, F_s \in \ff^*_i(T_s \sqcup X)$. Hence, $\ff$ contains a sunflower $X \sqcup T_1 \sqcup F_1, \ldots, X \sqcup T_s \sqcup F_s$ with $s$ petals and the core of size $t - 1$, a contradiction.

    Thus, $\cS_{i - 1} \sqcup \ff^*_i(D)$ does not contain a sunflower with $s$ petals and the core of size at most $t - 1$. Therefore, $|\ff^*_i(D)| \le \phi_i$, and
    \begin{align}
    \label{eq: upper bound on Fi star}
        |\ff^*_i| \le \phi_i \cdot \binom{n - |\bfS|}{k - t- i} + O_{s, k}(n^{- i - 1}) \binom{n - t}{k - t}.
    \end{align}

    \noindent \textbf{Step 6. Deriving the value of $\phi_i$.} We claim that 
    \begin{align}
    \label{eq: phi-i equality}
        \phi_i + \sum_{j = 0}^{i - 1} |\cS^{(t + j)}_{i - 1}| \binom{T -|\bfS|}{i - j} = \widetilde{\phi}_i(\cS_*),
    \end{align}
    where $\widetilde{\phi}_j, \cS_*$ were defined in Section~\ref{section: extremal families}. First, we show that 
    \begin{align*}
        \phi_i + \sum_{j = 0}^{i - 1} |\cS^{(t + j)}_{i - 1}| \binom{T - |\bfS|}{i - j} \le \widetilde{\phi}_i(\cS_*).
    \end{align*}
    Let $\cQ$ be any family for which the maximum in the definition~\eqref{eq: definition of phi i} of $\phi_i$ is attained. Then, we have
    \begin{align}
        \widetilde{\phi}_i(\cS_{i - 1} \sqcup \cQ) = \phi_i + \sum_{j = 0}^{i - 1} |\cS^{(t + j)}| \binom{T - |\bfS|}{i - j}. \label{eq: phi i upper bound part}
    \end{align}
    By induction, the family $\cS_{i - 1}$ satisfies property~\ref{item: induction max lexicographical order}, and so we have $\widetilde{\phi}_j (\cS_{i - 1} \sqcup \cQ) = \widetilde{\phi}_j(\cS_{i - 1}) = \widetilde{\phi}_j(\cS_*)$ for all $j \in [i - 1]$. By the lexicographic maximality of $\widetilde{\phi}(\cS_*)$, we have $\widetilde{\phi}_{i}(\cS_{i - 1} \sqcup \cQ) \le \widetilde{\phi}_i(\cS_*)$, so~\eqref{eq: phi i upper bound part} follows. It remains to prove that
    \begin{align*}
        \phi_i + \sum_{j = 0}^{i - 1} |\cS^{(t + j)}_{i - 1}| \binom{T - |\bfS|}{i - j} \ge \widetilde{\phi}_i(\cS_*).
    \end{align*}
    
    Recall the definition \eqref{eq: extremal example refined} of $\ff_*$. Since $\ff$ is extremal, we have
    \begin{align}
    \label{eq: extremality condition}
        |\ff| \ge |\ff_*| = \sum_{r = 0}^i \widetilde{\phi}_r(\cS_*) \binom{n - T}{k - t-r} + O_{s,k}(n^{- i - 1}) \binom{n}{k - t},
    \end{align}
    where we used Proposition~\ref{proposition: phi representation} for the equality. Using  Proposition~\ref{proposition: phi representation} for $\cS_{i-1}$ with $q=t+i$,
we also have 
\begin{align}\label{eqsi-1}
    |\ff_{\cS_{i - 1}}| = \sum_{r = 0}^{i} \widetilde{\phi}_r(\cS_{i - 1}) \binom{n - T}{k - t - r}  + O_{s,k} (n^{-i - 1}) \binom{n}{k - t}
\end{align}

Next, we combine the above. 
{\small \begin{align}
    & \sum_{r = 0}^i \widetilde{\phi}_i(\cS_*) \binom{n - T}{k - t - r} + O_{s,k}(n^{-i - 1}) \binom{n}{k - t} \overset{\eqref{eq: extremality condition}}{\le}|\ff|\overset{\eqref{eq: removing raminder on i-th layer}}\le |\ff_{\cS_{i - 1}}| + |\ff^*_i| + O_{s,k} (n^{-i - 1}) \binom{n}{k - t} \label{eq: lower bound on F-S and F-i star}\\
    & \qquad \overset{\eqref{eq: upper bound on Fi star},\eqref{eqsi-1}}{\le} \sum_{r = 0}^{i} \widetilde{\phi}_r(\cS_{i - 1}) \binom{n - T}{k - t - r} + \phi_i  \binom{n}{k - t - i}  + O_{s,k}(n^{- i - 1}) \binom{n}{k - t} \nonumber \\
    & \qquad = \sum_{r = 0}^{i - 1} \widetilde{\phi}_r(\cS_{i - 1}) \binom{n - T}{k - t - r} + \left (\phi_i +  \widetilde{\phi}_i(\cS_{i - 1}) \right ) \binom{n}{k - t - i} + O_{s,k}(n^{- i - 1}) \binom{n}{k - t}, \nonumber
\end{align}}
Using again that $\widetilde{\phi}_j(\cS_{i - 1}) = \widetilde{\phi}_{j}(\cS_*)$ for $j \in [i - 1]$ due to property~\ref{item: induction max lexicographical order} and
\begin{align}
   \label{eqphi} \widetilde{\phi}_i(\cS_{i - 1}) = \sum_{j = 0}^{i - 1} |\cS^{(t + j)}_{i - 1}| \binom{T - |\bfS|}{i - j},
\end{align}
we deduce~\eqref{eq: phi-i equality}, provided $n$ is large enough. 

Finally, we can lower bound $|\ff^*_i|.$
\begin{align}
    |\ff^*_i| & \overset{\eqref{eq: lower bound on F-S and F-i star},\eqref{eqphi}}{\ge} \left ( \widetilde{\phi}_i(\cS_*) - \sum_{j = 0}^{i - 1} |\cS_{i - 1}^{(t + j)}| \binom{T - |\bfS|}{i - j} \right ) \binom{n - T}{k - t- i} - O_{s,k} \left ( n^{-i - 1} \right ) \binom{n}{k - t} \nonumber \\
    & \overset{\eqref{eq: phi-i equality}}{=} \phi_i \binom{n - |\bfS|}{k - t- i} - O_{s,k} (n^{- i - 1}) \binom{n}{k - t}. \label{eq: Fi star lower bound using phi}
\end{align}

\noindent \textbf{Step 7. Applying isoperimetry.} This part follows the proof of Theorem~\ref{theorem: stability in Erdos-Duke}. Define
\begin{align*}
    \cD_i^* = \bigcap_{S \in \widetilde{\cS}_i} \cD^*_{i, S}.
\end{align*}
Note that $\ff_i^*(D)$ is non-empty iff $D \in \cD^*_i$ by~\eqref{eq: F star i definition}. We may refine the bound~\eqref{eq: bound with fat D i of S i}:
\begin{align*}
    |\ff^*_i| & = \sum_{D \in \cD_{i - 1}} |\ff^*_i(D)| + O_{s,k}(1/n) \binom{n}{k - t- i} \\
    & = \sum_{D \in \cD_{i - 1} \cap \cD^*_i} |\ff^*_i(D)| + O_{s,k}(1/n) \binom{n}{k - t- i}.
\end{align*}
Recall that $|\ff^*_i(D)| \le \phi_i$. Let 
\begin{align*}
    \widetilde{\cD}_i = \{D \in \cD_i^* \cap \cD_{i - 1} \mid |\ff^*_i(D)| = \phi_i\}.
\end{align*}
From the lower bound~\eqref{eq: Fi star lower bound using phi}, we derive
\begin{align*}
    \phi_i \cdot |\widetilde{\cD}_i| + (\phi_i - 1) \cdot \left | \binom{[n] \setminus \bfS}{k - t- i} \setminus \widetilde{\cD}_i \right | \ge \phi_i \binom{n - |\bfS|}{k - t- i} - O_{s,k}(1/n) \cdot \binom{n}{k - t- i},
\end{align*}
and so
\begin{align}
    |\widetilde{\cD}_i| \ge \binom{n - |\bfS|}{k - t- i} - O_{s,k}(1/n) \binom{n}{k - t - i}. \label{eq: large measure of tilde D i}
\end{align}

Consider the Jonson graph $J_{[n] \setminus \bfS, k - t- i}$ with the vertex set $\binom{[n] \setminus \bfS}{k - t-i}$ and two sets $D_1, D_2$ being adjacent iff $|D_1 \cap D_2| = k - t-i - 1$. As in the proof of Theorem~\ref{theorem: stability in Erdos-Duke}, we claim that if $D_1, D_2 \in \widetilde{\cD}_i$ are adjacent, then $\ff^*_i(D_1) = \ff^*_i(D_2)$. Otherwise $|\ff^*_i(D_1) \cup \ff^*_i(D_2)| \ge \phi_i + 1$, so by the definition~\eqref{eq: definition of phi i} of $\phi_i$, there exists a sunflower $T_1, \ldots, T_s \in \cS_{i - 1} \cup \ff^*_i(D_1) \cup \ff^*_i(D_2)$ with $s$ petals and the core of size at most $t - 1$. WLOG, we assume that $T_1, \ldots, T_{p} \in \cS_{i - 1}$ and $T_{p + 1}, \ldots, T_s \in \ff^*_i(D_1) \cup \ff^*_i(D_2)$. Let $\ell = t - 1- |\cap_i T_i|$ and consider two cases.

\textbf{Case 1. We have $\ell = 0$.} Then, we have $|\ff(T_j, \bfS)| \ge (1 - O_{s,k}(1/n)) \binom{n - |\bfS|}{k - |T_j|}$ for $j \in [p]$ by property~\ref{item: fat restrictions over S} of the lemma and $|\ff^*_i(T_j)| \ge sk \binom{n - |\bfS|}{k - t- i}$ if $j > p$ by the definition~\eqref{eq: tilde S_i definition} of $\widetilde{\cS}_i$. Hence, by Proposition~\ref{proposition: cross matching}, there exist disjoint sets $F_1 \in \ff(T_1, \bfS), \ldots, F_s \in \ff^*_i(T_s) = \ff_i^*(T_s, \bfS)$, so $\ff$ contains a sunflower $F_1 \sqcup T_1, \ldots, F_s \sqcup T_s$, a contradiction.

\textbf{Case 2. We have $\ell > 0$.} Since $k \ge (q + t + 1)(t - 1) \ge t + i + t - 1$, there exists $X \subset D_1 \cap D_2$ of size $\ell$. Since $D_1, D_2 \in \widetilde{\cD}_i \subset \cD_i^* \cap \cD_{i - 1}$, we have $|\ff(T_j \sqcup X, \bfS \sqcup X)| \ge sk \binom{n - |\bfS|}{k - |T_j|- |X| - 1}$ for $j \in [p]$ by the definition of~\eqref{eq: Di-1 definition} of $\cD_{i - 1}$ and $|\ff_{i}^*(T_j \sqcup X, \bfS \sqcup X)| \ge sk \binom{n - |\bfS| - |X|}{k - t - i -|X|- 1}$ for $j > p$ by the definition of $\cD^*_i$ (cf.~\eqref{eq: D stap iS definition}). By Proposition~\ref{proposition: cross matching}, there are disjoint sets $F_1 \in \ff(T_j \sqcup X, \bfS \sqcup X), \ldots, F_s \in \ff^*_i(T_s \sqcup X, \bfS \sqcup X)$. Hence, $\ff$ contains a sunflower with $s$ petals and the core of size $t - 1$, a contradiction.

The above implies $\ff_i^*(D_1) = \ff_i^*(D_2)$, so for each connected component $\cC$ of the subgraph $J_{[n] \setminus \bfS, k - t - i}[\widetilde{\cD}]$ there exists a family $\cQ$ such that $\ff_i^*(D) = \cQ$ for any $D \in \cC$, and $\cS_{i - 1} \sqcup \cQ$ does not contain a sunflower with $s$ petals and the core of size at most $t - 1$. Let $\bbC[\widetilde{\cD}_i]$ be the set of connected components of $J_{[n] \setminus \bfS, k - t- i}[\widetilde{\cD}]$ and let $\widetilde{\cC}_i$ be the largest connected component of $J_{[n] \setminus \bfS, k - t- i}[\widetilde{\cD}]$. For each $\cC \in \bbC[\widetilde{\cD}_i] \setminus \widetilde{\cC}_i$, edges of $\partial \cC$ connect vertices of $\cC$ and $\binom{[n] \setminus \bfS}{k - t - i} \setminus \widetilde{\cD}_i$. Also, we have $|\cC| \le \frac{1}{2} \binom{n - |\bfS|}{k - t- i}$. Recall that the spectral gap of $J_{n - |\bfS|, k - t- i}$ is $(n - |\bfS|)$ (see~\eqref{eq: Johnson graph spectral gap}), so Lemma~\ref{lemma: Cheeger inequality} yields:
\begin{align*}
    \frac{n - |\bfS|}{2} |\widetilde{\cD}_i \setminus \widetilde{\cC}_i| & = \frac{n - |\bfS|}{2} \sum_{\cC \in \bbC[\widetilde{\cD}_i] \setminus \{\widetilde{\cC}_i\}} |\cC| \le \sum_{\cC \in \bbC[\widetilde{\cD}_i] \setminus \{\widetilde{\cC}_i\}} |\partial \cC| \\
    & \le \Delta(J_{[n] \setminus \bfS, k - t -i}) \cdot \left | \binom{[n] \setminus \bfS}{k - t - i} \setminus \widetilde{\cD}_i \right |,
\end{align*} 
where $\Delta(J_{[n] \setminus \bfS, k - t- i}) = (n - |\bfS|)(k - t - i)$ is the largest degree of $J_{[n] \setminus \bfS, k - t- i}$. Using~\eqref{eq: large measure of tilde D i}, we derive
\begin{align}
\label{eq: largest i-th component is fat}
    |\widetilde{\cD}_i \setminus \widetilde{\cC}_i| = O_{s,k}(1/n) \binom{n- |\bfS|}{k - t- i},
\end{align}
so $|\widetilde{\cC}_i| \ge (1 - O_{s,k}(1/n)) \binom{n - |\bfS|}{k - t - i}$. Since $\ff_i^*(D)$ is the same family for all $D \in \widetilde{\cC}_i$, set $\cT_i = \ff^*_i(D)$ for any $D \in \widetilde{\cC}_i$. Then, set $\cS_i = \cS_{i - 1} \sqcup \cT_i$. Note that by construction, $\cS_i$ does not contain a sunflower with $s$ petals and the core of size at most $t - 1$, so property~\ref{item: induction sunflower restriction} holds; moreover, $\cS_i \subset 2^{\bfS}$, so property~\ref{item: induction support restriction} follows. Since $\cC_i \subset \widetilde{\cD}_i$, we have $|\cT_i| = \phi_i$, so
\begin{align*}
    \widetilde{\phi}_i(\cS_i)& = |\cT_i| + \sum_{r = 0}^{i - 1} |\cS_i^{(t + j)}| \binom{T - |\bfS|}{i - r} \\
    & = \phi_i + \sum_{r = 0}^{i - 1} |\cS_{i - 1}^{(t + j)}| \binom{T - |\bfS|}{i - r} = \widetilde{\phi}_i(\cS_*)
\end{align*}
by~\eqref{eq: phi-i equality}. Since $\widetilde{\phi}_{j}(\cS_i) = \widetilde{\phi}_j(\cS_{i - 1}) = \widetilde{\phi}_j(\cS_*), j < i,$ by the lexicographic maximality of $(\widetilde{\phi}_j(\cS_{i - 1}))_{j = 0}^{i - 1}$ from property~\ref{item: induction max lexicographical order}, we derive that the vector
\begin{align*}
    (\widetilde{\phi}_0(\cS_{i}), \ldots, \widetilde{\phi}_i(\cS_i))
\end{align*}
is lexicographically maximal among families $\cS$ satisfying properties~\ref{item: induction sunflower restriction},\ref{item: induction support restriction}. Hence, property~\ref{item: induction max lexicographical order} holds for $\cS_i$.

\noindent \textbf{Step 8. Verifying property~\ref{item: fat restrictions over S}.} Next, we prove that $|\ff(T, \bfS)| \ge (1 - O_{s,k}(1/n)) \binom{n - |\bfS|}{k - t- i}$ for all $T \in \cT_i$. It is enough to bound $\ff^*_i(T)$ from below. Since $|\widetilde{\cD}_i| \ge (1 - O_{s,k}(1/n)) \binom{n - |\bfS|}{k - t- i}$ due to~\eqref{eq: large measure of tilde D i}  and  $|\widetilde{\cD}_i \setminus \widetilde{\cC}_i| \le O_{s,k}(1/n) \binom{n - |\bfS|}{k - t- i}$ due to~\eqref{eq: largest i-th component is fat}, we have
\begin{align*}
    |\ff(T, \bfS)| \ge |\ff_i^*(T)| \ge |\widetilde{\cC}_i| = (1 - O_{s,k}(1/n)) \binom{n - |\bfS|}{k - t -i},
\end{align*}
so $\cS_i$ possesses property~\ref{item: fat restrictions over S}. It remains to bound $|\cR_i|$, where $\cR_i$ is defined by~\eqref{eq: definition of R i}.

\noindent \textbf{Step 9. Bounding $|\cR_i|$.} We will use Proposition~\ref{proposition: bounds on same uniformity remainder}. Consider a set $F \in \cR_i$. We claim that for some $S_1, \ldots, S_{s - 1} \in \cS_{i}$, sets $F, S_1, \ldots, S_{s - 1}$ form a sunflower with $s$ petals and the core of size at most $t - 1$. Suppose that there are no such sets $S_1, \ldots, S_{s - 1}$. We have $F \cap \bfS \not \in \cS_{i}$ and $|F \cap \bfS| \le t + i$. If $|F \cap \bfS| \le t$, choose any $X \subset F \setminus \bfS$ of size $t - |F \cap \bfS|$.  Then, $X \sqcup (F \cap \bfS)$ can be added to $\cS^{(t)}_i$, contradicting the maximality of $\widetilde{\phi}_0(\cS_i)$. Hence, $|F \cap \bfS| > t$. Put $j = |F \cap \bfS| - t$, $j \le i$ and set $\cQ = \cS_i \sqcup \{F \cap \bfS\}$. Note that $\cQ$ satisfies properties~\ref{item: induction sunflower restriction},\ref{item: induction support restriction}. Then, $\widetilde{\phi}_k(\cQ) = \widetilde{\phi}_k(\cS_i)$ for $k < j$ and
\begin{align*}
    \widetilde{\phi}_j(\cQ) = |\cS_i^{(t + j)}| + 1 + \sum_{r = 0}^{j - 1} |\cS_i^{(t + r)}| \binom{T - |\bfS|}{j - r} > \widetilde{\phi}_j(\cS_i),
\end{align*}
contradicting the lexicographic maximality of $(\widetilde{\phi}_0(\cS_i), \ldots, \widetilde{\phi}_i(\cS_i))$. Hence, there are $S_1, \ldots, S_{s - 1} \in \cS_i$ such that $F, S_1, \ldots, S_{s - 1}$ form a sunflower with $s$ petals and the core of size at most $t - 1$. It implies $\cR_i \subset \widetilde{\cR}_i$, where $\widetilde{\cR}_i$ is defined in Proposition~\ref{proposition: bounds on same uniformity remainder}, so $|\cR_i| = O_{s,k} \left ( n^{- \frac{k - t (t - 1)}{t - 1}} \right ) \binom{n - t}{k - t}$.

Choosing $\cS = \cS_{q - t}$ and noting that $\cR = \cR_{q - t}$, $\widetilde{\cR} = \widetilde{\cR}_{q - t}$, we obtain the lemma.
\end{proof}

Equipped with this lemma, we are ready to prove Theorem~\ref{theorem: structural duke-erdos}.
\begin{proof}[Proof of Theorem~\ref{theorem: structural duke-erdos}]

Apply Lemma~\ref{lemma: induction over layers} with $q = 2t - 1$. Then we obtain the  families $\cS$ and $\ff_{\cS} = \{F \in \binom{[n]}{k} \mid F \cap \support \cS^{(t)} \in \cS\}$ such that
\begin{align*}
    |\ff \setminus \ff_{\cS}| = O_{s,k}(n^{-t}) \binom{n}{k - t}.
\end{align*}
Put $\cT = \cS^{(t)}$. By property~\ref{item: induction max lexicographical order} from Lemma~\ref{lemma: induction over layers}, $|\cT| = \phi(s, t)$. We claim that
\begin{align*}
    \ff_{\cT} = \left \{ F \in \binom{[n]}{k} \mid F \cap \support \cT \in \cT \right \}
\end{align*}
is a subfamily of $\ff$. Suppose that $\ff_{\cT}$ is not a subfamily of $\ff$. By the extremality of $\ff$, it means that $\ff_{\cT} \cup \ff$ contains a sunflower $F_1, \ldots, F_s$ with $s$ petals and the core of size $t - 1$. WLOG, we assume that for some $p \in [s - 1]$ we have $F_1, \ldots, F_p \in \ff \setminus \ff_{\cT}$ and $F_{p+1},\ldots, F_s\in \ff_{\mathcal T}$. Then, we consider two cases.

\textbf{Case 1. We have $\cap_i F_i \subset \support \cT$.} Put $\cG_i = \ff \left (F_i \cap \support \cT, \support \cT \cup \bigcup_{j = 1}^p F_j \right )$ for $i = p + 1, \ldots, s$. Then, we have
\begin{align*}
    |\cG_i| = |\ff(F_i \cap \support \cT,  \support \cT)| - pk \binom{n - |\support \cT \cup \cup_{j = 1}^p F_j|}{k - t - 1} = ( 1- O_{s,k}(1/n)) \binom{n}{k - t},
\end{align*}
where we used that $F_i \cap \support \cT \in \cT$ for $i = p + 1, \ldots, s$, and property~\ref{item: fat restrictions over S} from Lemma~\ref{lemma: induction over layers}. Due to Proposition~\ref{proposition: cross matching}, there exist disjoint sets $G_{p + 1} \in \cG_{p + 1}, \ldots, G_{s} \in \cG_s$. Hence, $\ff$ contains a sunflower $F_1, \ldots, F_p, G_{p + 1} \sqcup (F_{p + 1} \cap \support \cT), \ldots, G_s \sqcup (F_s \cap \support \cT)$ with $s$ petals and the core of size $t - 1$, a contradiction.

\textbf{Case 2. We have $\ell = |\cap_i F_i \setminus \support \cT| > 0$.} Fix $X = \cap_i F_i \setminus \support \cT$. We claim that for some $i = p + 1, \ldots, s$, we have
\begin{align}
\label{eq: small measure for some F_i cup X}
    |\ff((F_i \cap \support \cT) \cup X, \support \cT \cup X)| < sk \binom{n - |\support \cT| - \ell}{k - t - \ell - 1}.
\end{align}
Assume that the opposite holds. Set $\cG_i = \ff((F_i \cap \support \cT) \cup X, \support \cT \cup X \cup \bigcup_{j = 1}^p F_j) $. We have
\begin{align*}
    |\cG_i| & \ge |\ff((F_i \cap \support \cT) \cup X, \support \cT \cup X)| - pk \binom{n - |\support \cT| - \ell}{k -t - \ell - 1} \\
    & \ge (s - p) k \binom{n - |\support \cT \cup \bigcup_{j = 1}^p F_j|}{k - t- \ell - 1}.
\end{align*}
Proposition~\ref{proposition: cross matching} implies that there are disjoint sets $G_{p + 1} \in \cG_{p + 1}, \ldots, G_s \in \cG_{s}$. Hence, $\ff$ contains a sunflower $F_1, \ldots, F_p, (F_{p + 1} \cap \support \cT) \sqcup X \sqcup G_{p + 1}, \ldots, (F_s \cap \support \cT) \sqcup X \sqcup G_s$ with $s$ petals and the core of size $t - 1$, a contradiction.

Thus, for some $i = p + 1, \ldots, s$, we have~\eqref{eq: small measure for some F_i cup X}. We get the following inequality. (Recall that $|(F_i\cap \support \cT)\cup X|=t+\ell$.)
\begin{align*}
    |\ff| & \le |\ff_{\cS}| - |\ff_{\cS}((F_i \cap \support \cT) \cup X, \support \cT \cup X)| \\
    & \quad + sk \binom{n - |\support \cT| - \ell}{k - t - \ell - 1} + O_{s,k}(n^{-t}) \binom{n}{k - t} \\
    & = |\ff_\cS| - \binom{n - |\support \cT| - \ell}{k - t - \ell} + O_{s,k}(n^{- \ell - 1}) \binom{n}{k - t}.
\end{align*}
Since $|\ff| \ge |\ff_{\cS}|$, we have
\begin{align*}
    \binom{n - |\support \cT| - \ell}{k - t - \ell} = O_{s,k}(n^{-\ell - 1}) \binom{n}{k - t},
\end{align*}
a contradiction for large enough $n$.

Hence, $\ff_{\cT} \cup \ff$ does not contain a sunflower with $s$ petals and the core of size $t - 1$, so $\ff_{\cT} \subset \ff$ by the extremality of $\ff$.

It remains to prove that for each $F \in \ff \setminus \ff_{\cT}$ we have $|F \cap \support \cT| \ge t + 1$. Suppose that the opposite holds, i.e., for some $F \in \ff \setminus \ff_{\cT}$ we have $|F \cap \support \cT| \le t$. We claim that there exist sets $T_1, \ldots, T_{s - 1} \in \cT$ such that $F \cap \support \cT, T_1, \ldots, T_{s - 1}$ form a sunflower with $s$ petals and the core of size at most $t - 1$. Indeed, otherwise, one can choose a set $X$ of size $t$ such that $F \cap \support \cT \subset X \subset F$, and add it to $\cT$, enlarging it. It contradicts $|\cT| = \phi(s, t)$ and the definition of $\phi(s, t)$. Next, fix any $Y \subset F \setminus \support \cT$ of size $t - 1 - |F \cap \bigcap_i T_i|$. We have
\begin{align*}
    |\ff_{\cT}(T_i \cup Y, \support \cT \cup F)| = \binom{n - |\support \cT \cup F|}{k - |T_i| - |Y|} \ge (s - 1)k \binom{n - |\support \cT \cup F|}{k - |T_i| - |Y| - 1},
\end{align*}
so there are disjoint sets $F_1 \in \ff_{\cT}(T_1 \cup Y, \support \cT \cup F), \ldots, F_{s - 1} \in \ff_{\cT}(T_{s - 1} \cup Y, \support \cT \cup F)$ due to Proposition~\ref{proposition: cross matching}. Hence, $\ff$ contains a sunflower $F, F_1 \sqcup Y \sqcup T_1, \ldots, F, F_{s - 1} \sqcup Y \sqcup T_{s - 1}$ with $s$ petals and the core of size $t - 1$, a contradiction.
\end{proof}

\subsection{Proof of Proposition~\ref{proposition: bounds on same uniformity remainder}}
\label{subsection: proof of kruskal-katona type bound for remainder}

Using notation of Lemma~\ref{lemma: induction over layers}, we set $\bfS = \support \cS^{(t)}_i$. Decomposing
\begin{align*}
    |\widetilde{\cR}_i| = \sum_{X \subset \bfS} |\widetilde{\cR}_i(X, \bfS)|,
\end{align*}
we conclude that for some $X \subset \bfS$, we have
\begin{align}
\label{eq: one restriction upper bound}
    |\widetilde{\cR}_i| \le 2^{|\bfS|} |\widetilde{\cR}_i(X, \bfS)|. 
\end{align}

Consider a set $F \in \widetilde{\cR}_i(X, \bfS)$. By the definition, there exist sets $T_1, \ldots, T_{s - 1} \in \cS_i$ such that $F \sqcup X, T_1, \ldots, T_{s - 1}$ form a sunflower with $s$ petals and the core of size at most $t - 1$. Then, $X = (F \sqcup X) \cap \bfS, T_1, \ldots, T_{s - 1}$ form a sunflower with $s$ petals and the core of size at most $t - 1$. Denote  its core $X \cap \bigcap_{j = 1}^{s - 1} T_j$ by $C$ and define $\ell = t - 1- |C|$. Consider two cases.

    \textbf{Case 1. We have $\mathbf{\ell = 0}$.} We claim that $\widetilde{\cR}_i(X, \bfS) = \varnothing$ in this case. Let $F'$ be any set in $\widetilde{\cR}_i(X, \bfS)$ and fix $F = X \cup F'$. For brevity, we denote $\cG_j = \ff(T_j, \bfS)$ for $j \in [s - 1]$. Since $\cS_i$ satisfies property~\ref{item: fat restrictions over S} of Lemma~\ref{lemma: induction over layers}, we have
    \begin{align*}
        |\cG_j| \ge (1 - O_{s,k}(1/n)) \binom{n - |\bfS|}{k - |T_j|}.
    \end{align*}
    For each $j \in [s - 1]$, it implies
    \begin{align*}
        |\cG_j(\varnothing, F')| \ge |\cG_j| - |F'| \binom{n - |\bfS|}{k - |T_j| - 1} = (1 - O_{s,k}(1/n)) \binom{n - |\bfS| - |F'|}{k - |T_j|}.
    \end{align*}
    Then, Proposition~\ref{proposition: cross matching} implies that there exist disjoint sets $F_1 \in \cG_1(\varnothing, F'), \ldots, F_{s - 1} \in \cG_{s - 1}(\varnothing, F')$, so $\ff$ contains a sunflower $F, F_1 \sqcup T_1, \ldots, F_{s - 1} \sqcup T_{s - 1}$ with $s$ petals and the core of size $t - 1$, a contradiction.

    \textbf{Case 2. We have $\mathbf{\ell > 0}$.} For each $j \in [s - 1]$, define
    \begin{align}
    \label{eq: fat shadows}
        \hh_j = \left \{H \in \partial_{\ell}\cG_j \mid |\cG_j(H)| \ge sk \binom{n -  |\bfS| - \ell}{k - |T_j| - \ell - 1}  \right \}.
    \end{align}
    By double counting, we have
    \begin{align*}
        \binom{k - |T_j|}{\ell} |\cG_j| & = \sum_{H \in \partial_{\ell} \cG_j} |\cG_j(H)| \le |\hh_j| \cdot \binom{n - |\bfS| - \ell}{k - |T_j| - \ell} \\
        & \quad + sk \cdot (|\partial_\ell \cG_j| - |\hh_j|) \binom{n - |\bfS| -  \ell}{k - |T_j| - \ell - 1}.
    \end{align*}
    Since $|\cG_j| = (1 - O_{s,k}(1/n)) \binom{n - |\bfS|}{k - |T_j|}$ and $|\partial_\ell \cG_j| \le \binom{n - |\bfS|}{\ell}$, we obtain
    \begin{align*}
        |\hh_j| & \ge \left (1 - O_{s,k}(1/n) \right ) \frac{\binom{k - |T_j|}{\ell} \binom{n - |\bfS|}{k - |T_j|}}{\binom{n - |\bfS| - \ell}{k - |T_j|- \ell}} \nonumber \\
        & = \left (1 - O_{s,k}(1/n) \right ) \binom{n - |\bfS|}{\ell}.
    \end{align*}
    Since $\hh_j \subset \binom{[n]\setminus \bfS}{\ell}$, by the union bound, we have
    \begin{align}
    \label{eq: fat intersetion of shadows}
        \Big|\bigcap_{j \in [s - 1]} \hh_j\Big| = \left (1 - O_{s,k}(1/n) \right ) \binom{n - |\bfS|}{\ell}.
    \end{align}
    We claim that $\partial_\ell (\widetilde{\cR}_i(X, \bfS))$ and $\cap_j \hh_j$ do not intersect. Suppose that there exists $H \in \partial_\ell (\widetilde{\cR}_i(X, \bfS)) \cap \bigcap_i \hh_i$. Fix a set $F \in \widetilde{\cR}_i(X \cup H, \bfS \cup H)$. Then, for each $j \in [s - 1]$ by the definition of $\hh_j$, we have
    \begin{align*}
        |\cG_j(H, H \sqcup F)| & \ge sk \binom{n - |\bfS| - \ell}{k - |T_j| - \ell - 1} - |F| \binom{n - |\bfS| - \ell}{k - |T_j| - \ell - 1} \\
        & \ge (s - 1) k \binom{n - |\bfS| - \ell}{k - |T_j| - \ell - 1}.
    \end{align*}
    By Proposition~\ref{proposition: cross matching}, there are disjoint sets $F_1 \in \cG_1, \ldots, F_{s - 1} \in \cG_{s - 1}$. Hence, $F$ contains a sunflower $X \sqcup H \sqcup F, T_1 \sqcup H \sqcup F_1, \ldots, T_{s - 1} \sqcup H \sqcup F_{s - 1}$ with $s$ petals and the core of size $t - 1$, a contradiction. Hence, $\partial_\ell (\widetilde{\cR}_i(X, \bfS))$ and $\cap_j H_j$ are disjoint, and thus
    \begin{align*}
        |\partial_\ell (\widetilde{\cR}_i(X, \bfS))| = O_{s,k}(1/n) \binom{n - |\bfS|}{\ell}
    \end{align*}
    due to~\eqref{eq: fat intersetion of shadows}. Let $y$ be a real number such that $|\partial_\ell (\widetilde{\cR}_i(X, \bfS))| = \binom{y}{\ell}$. Then, $y = O_{s,k}(n^{1 - 1/\ell})$, and, using the Kruskal--Katona Theorem~\ref{theorem: kruskal--katona theorem}, we get
    \begin{align*}
        |\widetilde{\cR}_i(X, \bfS)| \le \binom{y}{k - |X|} = O_{s,k}(n^{k - |X| - (k - |X|) / \ell}) = O_{s,k}\left ( n^{t - |X| - (k - |X|)/\ell}\right ) \binom{n}{k - t}.
    \end{align*}
     Since $|X| - |X| / \ell \ge 0$, we have
    \begin{align*}
        |\widetilde{\cR}_i(X, \bfS)| = O_{s,k}\left (n^{-\frac{k - t\ell}{\ell}} \right ) \binom{n}{k - t} = O_{s,k} \left ( n^{-\frac{k - t (t - 1)}{t - 1}} \right ) \binom{n}{k - t},
    \end{align*}
    and the proposition follows.

\textbf{Remark. } The bound in this proposition is, in some sense, best possible. We can construct two families $\mathcal X,\mathcal Y$ of $k$-element subsets of $[n],$ such that $\mathcal X$ has size $(1-O(1/n)){n\choose k}$, $\mathcal Y$ has size $\Omega(n^{-k/\ell}){n\choose k}$, and such that their $\ell$-shadows are disjoint. (We simplified the parameters for this claim, but this lower bound matches the upper bound on $|\widetilde{\cR}_i(X, \bfS)|$ if one plugs in the correct uniformities.) The example is as follows: let $\mathcal X$ be the family of all $k$-set that intersect $[m]$ in at most $\ell-1$ elements, where $m = n^{1/\ell}$, and let $\mathcal Y$ be the family ${[m]\choose k}.$ It is easy to see that $\ell$-shadows of these two families are disjoint and the sizes are as claimed.
\subsection{Proof of Proposition~\ref{proposition: k-t-i fatness}}
\label{subsection: proof of k-t-i fatness}
    Consider some $S \in \cS_{i - 1}$. By the double counting, we have
    \begin{align*}
        \sum_{D \in \binom{n - |\bfS|}{k - t - i}} |\ff(S \sqcup D, \bfS \sqcup D)| = \binom{k - |S|}{k - t- i} |\ff(S, \bfS)| = (1 - O_{s,k}(1/n))\binom{k - |S|}{k - t- i} \binom{n - |\bfS|}{k - |S|}
    \end{align*}
    since $\cS_{i - 1}$ satisfies property~\ref{item: fat restrictions over S}. Define
    \begin{align*}
        \widehat{\cE}_{i - 1, p, S} = \left \{ X \in \binom{n - |\bfS|}{p} \mid |\ff(S \sqcup X, \bfS \sqcup X)| < sk \binom{n - |\bfS| - p}{k - |S| - p - 1}\right \}.
    \end{align*}
    If $D \not \in \cD_{i - 1}$, then there exist $p\in [t-1]$, $S\in \mathcal S_{i-1}$ and a set $E \in \widehat{\cE}_{i - 1, p, S}$ such that $E \subset D$. For each $E$ there are  $\binom{n - |\bfS| - p}{k - t - i - p}$ sets of $\binom{[n] \setminus \bfS}{k - t- i}$ that contain $E$. Hence, we have
    \begin{align}
        \label{eq: bound via hat caligraphic E}
        \left | \binom{[n] \setminus \bfS}{k - t - i}  \setminus \cD_{i - 1}\right | \le \sum_{S \in \cS_{i - 1}} \sum_{p = 1}^{t - 1} |\widehat{\cE}_{i - 1, p, S}| \binom{n - |\bfS| - p}{k - t - i - p}.
    \end{align}
    So, it remains to bound $\widehat{\cE}_{i - 1, p, S}$. For each $S \in \cS_{i - 1}$ and $p$, we have
    \begin{align*}
        \binom{k - |S|}{p} |\ff(S, \bfS)| & = \sum_{X \in \binom{[n] \setminus \bfS}{p}} |\ff(S \sqcup X, \bfS \sqcup X)| \\
        &  \le sk \binom{n - |\bfS| - p}{k - |S| - p - 1} |\widehat{\cE}_{i - 1, p , S}| \\
        & \quad + \binom{n - |\bfS| - p}{k - |S| - p} \left ( \binom{n - |\bfS|}{p} - |\widehat{\cE}_{i - 1, p, S}| \right ).
    \end{align*}
    Since $\cS_{i - 1}$ satisfies property~\ref{item: fat restrictions over S} of Lemma~\ref{lemma: induction over layers}, we have $|\ff(S, \bfS)| \ge (1 - O_{s,k}(1/n)) \binom{n - |\bfS|}{k - |S|}$. Using
    \begin{align*}
        \binom{n - |\bfS|}{k - |\bfS|} \binom{k - |S|}{p} = \binom{n - |\bfS|}{p} \binom{n - |\bfS| - p}{k - |S| - p}, 
    \end{align*}
    we get
    \begin{align*}
        |\widehat{\cE}_{i - 1, p, S}| (1 - O_{s,k}(1/n)) \binom{n - |\bfS| - p}{k - |S| - p} & = O_{s,k}(1/n) \binom{n - |S|}{p}\binom{n - |\bfS|}{k - |S| - p}, \\
        |\widehat{\cE}_{i - 1, p, S}| & = O_{s,k}(1/n) \binom{n - |\bfS|}{p}.
    \end{align*}
    Substituting the above into~\eqref{eq: bound via hat caligraphic E}, we get
    \begin{align*}
        \left | \binom{[n] \setminus \bfS}{k - t - i}  \setminus \cD_{i - 1}\right | & \le |\cS_{i - 1}| \cdot (t - 1) \cdot O_{s,k}(1/n) \binom{n - |\bfS|}{k - t- i} \\
        & = O_{s,k}(1/n)\binom{n - |\bfS|}{k - t- i},
    \end{align*}
    where we used $|\cS_{i - 1}| \le 2^{|\bfS|} \le 2^{t \phi(s, t)} = O_{s,k}(1)$. Thus, the proposition follows.

\subsection{Proof of Proposition~\ref{proposition: cross matching}}
\label{subsection: proof of cross-matching}

Let $p$ be the maximal integer such that there exist disjoint $F_1 \in \cA_1, \ldots, F_p \in \cA_p$. Assume that $p < s$. Then, we have
\begin{align*}
    |\cA_{p + 1}(\varnothing, \cup_{j = 1}^p F_j)| \ge \left (\sum_{j = 1}^s k_j \right ) \binom{m}{k_{p + 1} - 1} - \left ( \sum_{j = 1}^p |F_j| \right ) \binom{m}{k_{p + 1} - 1} > 0,
\end{align*}
so there exists $F_{p + 1} \in \cA_{p + 1}(\varnothing, \cup_{j = 1}^p F_j)$. Clearly, $F_1, \ldots, F_{p + 1}$ are disjoint, a contradiction. 

\section{Proof of Theorem~\ref{theorem: graph case Duke-Erdos extremal}}
\label{section: proof of theorem -- graph case}

\begin{proof}[Proof of Theorem~\ref{theorem: graph case Duke-Erdos extremal}]
 Let us apply Lemma~\ref{lemma: induction over layers} with $q = 3$, and obtain families $\cS \subset  \bigcup_{i = 0}^1 \binom{\support \cT}{2 + i},$ $\cR, \widetilde{\cR}$ such that
 \begin{align}
 \label{eq: F approximation in graph case}
     |\ff \setminus \ff_{\cS}| = O_{s,k}(n^{-2}) \binom{n}{k - 2}, \quad |\widetilde{\cR}| = O_{s,k} \left ( n^{-2} \right ) \binom{n}{k - 2}.
 \end{align}
 Due to Theorem~\ref{theorem: sunflower graph case}, $\cS^{(2)}$ can be represented as a disjoint union of two cliques $K_1 \sqcup K_2$, each of size $s$. 
 
 Suppose that there exists a set $F \in \ff$ such that $|F \cap K_i| = 1$ for some $i \in 1,2$. Fix these $F, i$. Let $v \in K_i$ be a vertex such that $\{v\} = F \cap K_i$. Consider edges $\{v, u\} \in E(G)$ for all $u \in V(K_i) \setminus \{v\}$. Then, $\{v\}$ and $\{v, u\},$ where $u \in V(K_i) \setminus \{v\}$, form a sunflower with $s$ petals and the core of size $1$. Since $F \cap K_i = \{v\}$, the sets $F \cap V(G)$ and $ \{v, u\},$ where $ u \in V(K_i) \setminus \{v\}$, form a sunflower with the same core. We claim that it implies for some $u \in V(K_i) \setminus \{v\}$
 \begin{align}
 \label{eq: small cardinality for some u}
     |\ff(\{v, u\}, \support \cS^{(t)} \cup F)| < sk \binom{n - |\support \cS^{(2)} \cup F|}{k - 3}.
 \end{align}
 Otherwise, Proposition~\ref{proposition: cross matching} implies that there are disjoint sets $F_u \in \ff(\{v, u\}, \support \cS^{(t)} \cup F)$ such that $F, F_u, u \in V(K_i) \setminus \{v\}$ form a sunflower with $s$ petals and the core of size $1$. But~\eqref{eq: small cardinality for some u} contradicts property~\ref{item: fat restrictions over S} of $\cS$ stated in Lemma~\ref{lemma: induction over layers}. Hence, there is no $F \in \ff$ such that $|F \cap K_i| = 1$ for some $i$.

 Next, we prove that each set $F \in \ff$, $F$ intersects $\support \cT$. Suppose that $F \cap (\support \cS^{(2)}) = \varnothing$. Let $e_1, \ldots, e_{s - 1}$ be any matching of $\cS^{(2)}$. Let $x \in F$ be any element of $F$. We claim that for some $e_j$, we have
 \begin{align*}
     |\ff(e_j \sqcup \{x\}, \support \cS^{(2)} \sqcup F)| < sk \binom{n - |\support \cS^{(2)}| - k}{k - 4}.
 \end{align*}
 Indeed, otherwise there are disjoint sets $F_i \in\ff(e_j \sqcup \{x\}, \support \cS^{(2)} \sqcup F)$ due to Proposition~\ref{proposition: cross matching}, and so $\ff$ contains a sunflower $F, e_1 \sqcup \{x\} \sqcup F_1, \ldots, e_{s - 1} \sqcup \{x\} \sqcup F_{s - 1}$ with $s$ petals and the core $\{x\}$ of size $1$. Due to~\eqref{eq: F approximation in graph case}, it implies
 \begin{align*}
     |\ff| & \le |\ff \setminus \ff_{\cS}| + |\ff_{\cS}| - |\ff_{\cS}(e_j \sqcup \{x\}, \support \cS^{(2)} \sqcup F)| + sk \binom{n - |\support \cS^{(2)}| - k}{k - 4} \\
     & = |\ff_{\cS}|  - \binom{n - |\support \cS^{(2)} \sqcup F|}{k - 3} + O_{s,k}(n^{-2}) \binom{n}{k - 2}.
 \end{align*}
 Since $\ff_{\cS}$ does not contain a sunflower with $s$ petals and the core of size at most $t - 1$, we have $|\ff| \ge |\ff_{\cS}|$, so
 \begin{align*}
     \binom{n - |\support \cS^{(2)} \sqcup F|}{k - 3} = O_{s,k}(n^{-2}) \binom{n}{k - 2}
 \end{align*}
 a contradiction for large enough $n$. Hence, each $F \in \ff$ intersects $\support \cS^{(2)}$, and $|F \cap K_i| \neq 1$ for each $i$. Considering $\cS^{(2)}$ as a graph $G = (\support \cS^{(2)}, \cS^{(2)})$, we get
 \begin{align*}
     \ff \subset \cG_* = \bigg \{F \in \binom{[n]}{k} \mid &  |F \cap V(G)| \ge 2 \\
    & \text{ and } \forall i \in \{1,2\} \text{ we have } |F \cap V(K_i)| \neq 1 \bigg \}.
 \end{align*}
 It remains to check that $\cG_*$ does not contain a sunflower with $s$ petals and the core of size $1$. Suppose that some $G_1, \ldots, G_s \in \cG_*$ form such sunflower. Then, $G_1 \cap V(G), \ldots, G_{s} \cap V(G)$ form a sunflower with the core of size at most $1$. By construction, for each $G_i \cap V(G)$ there is $e_i \in E(G)$ such that $e_i \subset G_i$. If $G_1 \cap V(G), \ldots, G_s \cap V(G)$ are disjoint, then there are disjoint $e_1, \ldots, e_s \in E(G)$. However, $E(G) = \cS^{(2)}$ does not contain a matching of size $s$. Hence, $G_1 \cap V(G), \ldots, G_s \cap V(G)$ form a sunflower with the core of size $1$. Let $\{v\} = V(G) \cap (\cap_i G_i)$. WLOG, assume that $v$ belongs to the first clique $K_1$. Then, $|G_i \cap K_1| \ge 2$ for each $i = 1, \ldots, s$ by the definition of $\cG_*$. Choose $u_i \in G_i \setminus \{v\}$ for each $i = 1, \ldots, s$. Since $G_1 \cap V(G), \ldots, G_s \cap V(G)$ form a sunflower with the core $\{v\}$, all $u_i$ must be distinct. But it means that $K_1$ contains $s + 1$ vertices, a contradiction.

 Hence, $\cG_*$ does not contain a sunflower with $s$ petals and the core of size $1$, so $\ff = \cG_*$ by the extremality of $\ff$.
\end{proof}

\section{Proof of Theorem~\ref{theorem: exact erdos-duke}}
\label{section: proof of the exact duke-erdos}

\begin{proof}[Proof of Theorem~\ref{theorem: exact erdos-duke}]
Apply Lemma~\ref{lemma: induction over layers} with $q = T$, and obtain a family $\cS$ possessing properties~\ref{item: induction sunflower restriction}-\ref{item: remainder bound} defined in this lemma. Since $q \ge T$, $\cS$ maximizes the vector-valued function $\widetilde{\phi}(\cS)$ in the lexicographical order due to property~\ref{item: induction max lexicographical order} from Lemma~\ref{lemma: induction over layers}, so $\cS$ admits properties~\ref{item: sunflower freeness S*}-\ref{item: lexcicographic maximality S*}. Thus, there exists a family $\cS_*$ (which is equal to $\cS$ obtained by Lemma~\ref{lemma: induction over layers}), such that 
\begin{align*}
    |\ff \setminus \ff_*| & = |\{F \in \ff \mid |F \cap \support \cS^{(t)}| \le T \text{ and } F \cap \support \cS^{(t)} \not \in \cS \}| =  |\cR| \\
    & = O_{s,k} \left ( n^{- \frac{k - t(t - 1)}{t - 1}} \right ) \binom{n}{k - t},
\end{align*}
where $\ff_*$ is defined by~\eqref{eq: extremal example refined} and $\cR, \widetilde{\cR}$ are defined in Lemma~\ref{lemma: induction over layers}. The first equality above is valid since the largest possible intersection with $\support S^{(t)}$ is at most $T$, and thus the sets fall in the remainder only if they belong to $\mathcal R.$ We claim that $\ff \subset \ff_*$. Suppose that there exists a set $F \in \ff \setminus \ff_*$. Then, $F \in \cR$, so $F \in \widetilde{\cR}$. Hence, there are $S_1, \ldots, S_{s - 1} \in \cS_*$ such that $F \cap \support \cS^{(t)}, S_1, \ldots, S_{s - 1}$ form a sunflower with $s$ petals and the core of size at most $t - 1$. Let $Y$ be any subset of $F \setminus \support \cS^{(t)}$ of size $t - 1 - |F \cap \bigcap_{i = 1}^{s - 1} S_i|$. We claim that for some $i = 1, \ldots, s- 1$, we have
\begin{align}
\label{eq: full erdos duke -- small S}
    |\ff(S_i \sqcup Y, \support \cS^{(t)} \cup F)| < (s - 1) k \binom{n - |(\support \cS^{(t)}) \cup F|}{k - |S_i| - |Y| - 1}.
\end{align}
Indeed, otherwise, we can apply Proposition~\ref{proposition: cross matching} to families $\cG_i = \ff(S_i \sqcup Y, \support \cT \cup F)$, $i = 1, \ldots, s- 1$ and obtain disjoint sets $F_1 \in \cG_1, \ldots, F_{s - 1} \in \cG_{s - 1}$. It implies that $\ff$ contains a sunflower $F, S_1 \sqcup Y \sqcup S_1, \ldots, S_{s - 1} \sqcup Y \sqcup F_{s - 1}$ with $s$ petals and the core of size $t - 1$, a contradiction.

Hence, we have~\eqref{eq: full erdos duke -- small S} for some $S_i$. Therefore, we have
\begin{align*}
    |\ff| & \le |\ff_*| - \binom{n - |(\support \cS^{(t)}) \cup F|}{k - |S_i| - |Y|} + (s - 1)k \binom{n - |(\support \cS^{(t)}) \cup F|}{k - |S_i| - |Y| - 1} + \\
    & \quad + O_{s,k} \left ( n^{- \frac{k - t(t - 1)}{t - 1}} \right ) \binom{n}{k - t}.
\end{align*}
On the other hand, we have $|\ff| \ge |\ff_*|$, so it implies
\begin{align*}
    O_{s,k} \left ( n^{- \frac{k - t(t - 1)}{t - 1}} \right ) \binom{n}{k - t} + (s - 1) k \binom{n - |(\support \cS^{(t)}) \cup F|}{k - |S_i| - |Y|  - 1} \ge \binom{n - |(\support \cS^{(t)}) \cup F|}{k - |S_i| - |Y|}.
\end{align*}
The left-hand side is $O_{s,k}\big(n^{k-t-\frac{k - t(t - 1)}{t - 1}}\big) = O_{s,k}\big(n^{k-\frac{k}{t - 1}}\big)=O\big(n^{k-T-t}\big)$, since $k \ge (T + t)(t - 1)$. The right-hand side is $\Omega_{k,s}\big(n^{k-|S_i|-|Y|}\big) = \Omega_{k,s}\big(n^{k-T-t+1}\big)$. This is a contradiction for large enough $n$, and so $\ff = \ff_*$ by the extremality of $\ff$.
\end{proof}

\printbibliography

\end{document}